\documentclass{mod}

\usepackage{titlesec}
\usepackage{titletoc}

\usepackage{caption}
\usepackage{subcaption}
\usepackage[all]{xy}
\usepackage{mathtools}

\usepackage{graphicx}
\usepackage{amsmath,amssymb,amsfonts}
\usepackage{mathrsfs}
\usepackage{amsthm}
\usepackage{rotating}
\usepackage{appendix}

\usepackage{textcomp}
\usepackage{hyperref}

\usepackage{tocloft}
\usepackage{tocbasic}

\newtheorem{thm}{Theorem}[section]

\theoremstyle{plain}
\newtheorem{lem}[thm]{Lemma}
\newtheorem{prop}[thm]{Proposition}

\newtheorem{defn-thm}[thm]{Definition-Theorem}

\newtheorem{defn-lem}[thm]{Definition-Lemma}
\newtheorem{defn-prop}[thm]{Definition-Proposition}

\newtheorem{defn}[thm]{Definition}

\newtheorem{assumption}[thm]{Assumption}

\theoremstyle{definition}

\newtheoremstyle{rmk}
{5pt}
{5pt}
{}
{}
{\itshape}
{}
{.5em}
{}

\theoremstyle{rmk}

\newtheorem{rmk}[thm]{Remark}

\titlecontents{section}
[0.72em]
{\scriptsize\bfseries\vspace{3pt}}%
{\thecontentslabel.\enspace}
{}
{\titlerule*[0.38pc]{.}\contentspage}%

\titlecontents{subsection}
[0em]
{\scriptsize \vspace{0pt}}%
{\qquad \quad \thecontentslabel.\enspace}
{}
{\titlerule*[0.5pc]{.}\contentspage}%

\titlecontents{subsubsection}
[0em]
{\footnotesize \vspace{1pt}}%
{\qquad \qquad \thecontentslabel.\enspace}
{}
{\titlerule*[0.5pc]{.}\contentspage}%

\titleformat{\subsection}[runin]{
	\bfseries\itshape\normalsize}{(\thesubsection) }{0em}{}[\mbox{ . } ]
\titlespacing{\subsection}
{0pt}
{2.25ex plus 1ex minus .2ex}
{0pt}

\newcommand{\val}{\mathrm{val}}

\newcommand{\id}{\mathrm{id}}

\newcommand{\mathds}[1]{\text{\usefont{U}{dsrom}{m}{n}#1}}
\newcommand{\one}{\mathds {1}}

\newcommand{\UD}{\mathscr{UD}}
\newcommand{\simud}{\stackrel{\mathrm{ud}}{\sim}}

\newcommand{\ev}{\mathrm{ev}}

\newcommand{\e}{\mathbf e}

\newcommand{\mi}{\mathfrak i}
\newcommand{\mI}{\mathfrak I}

\newcommand{\tri}{{\mathrm{tri}}}
\newcommand{\m}{{\mathfrak m}}

\newcommand{\M}{\mathfrak M}

\newcommand{\mc}{\mathfrak c}
\newcommand{\mC}{\mathfrak C}

\newcommand{\oi}{{[0,1]}}

\newcommand{\ab}{{[a,b]}}
\newcommand{\f}{\mathfrak f}
\newcommand{\F}{\mathfrak F}

\newcommand{\mH}{ H}
\newcommand{\h}{\mathfrak h}

\newcommand{\g}{\mathfrak g}

\newcommand{\HL}{H^*(L)}
\newcommand{\OL}{\Omega^*(L)}

\newcommand{\trop}{\mathfrak{trop}}

\DeclareMathOperator{\incl}{Incl}
\DeclareMathOperator{\eval}{Eval}

\DeclareMathOperator{\Spec}{Spec}

\DeclareMathOperator{\Hom}{Hom}

\DeclareMathOperator{\CC}{\mathbf{CC}}

\numberwithin{equation}{section}

\renewcommand{\theequation}{\thesection.\alph{equation}}

\begin{document}

	\title[
	Existence of pseudo-holomorphic disks via non-archimedean disk potentials
	]
	{{\fontsize{13}{15}\selectfont 
	Existence of pseudo-holomorphic disks via non-archimedean disk potentials
		}}

	\author{Hang Yuan}
	\date{}
	
\begin{abstract} {\sc Abstract:}  
We show that if a graded monotone Lagrangian $L_0$ has a non‑vanishing disk potential, then for every smooth isotopy $\{L_s\}_{s\in[0,1]}$ of Lagrangians starting from it and for every tame almost complex structure $J$, each $L_s$ bounds a $J$-holomorphic disk of Maslov index two.
The main input is a non-archimedean analytic potential function, defined as an invariant up to analytic isomorphisms, generalizing the classical disk potential of a monotone Lagrangian. The techniques are inspired by recent developments in the Strominger-Yau-Zaslow mirror construction via family Floer theory and non-archimedean geometry.
We also discuss applications such as recovering a simple case of Audin’s conjecture. 
\end{abstract}

\hypersetup{
	colorlinks=true,
	linktoc=all,
	citecolor=gray
}

\maketitle
	
\tableofcontents
%

\setlength{\parindent}{5.5mm}	\setlength{\parskip}{0em}
 
\section{Introduction}

\subsubsection*{\quad Background}
A well-known \textit{enumerative} invariant for a \textit{monotone} Lagrangian $L\subseteq X$ is the algebraic count of Maslov index 2 pseudo-holomorphic disks bounded by $L$, as pointed out in \cite{eliashberg1993unknottedness}.
There are a few applications.
For example, Auroux \cite{auroux2015infinitely} uses this invariant to prove that there are infinitely many monotone Lagrangian tori in $\mathbb R^6$, no two of which are related by Hamiltonian isotopies.
Vianna also uses it to identify an exotic Lagrangian torus in $\mathbb {CP}^2$ \cite{vianna2014exotic} and shows that $\mathbb{CP}^2$ contains inﬁnitely many non-isotopic monotone Lagrangian tori \cite{Vianna2016infinitely}.
Let's recall the definition based on \cite[\S 3.1]{auroux2015infinitely}:
Fix a homotopy class $\beta\in H_2(X,L)$ with its Maslov index $\mu(\beta)=2$.
Consider the moduli space of $J$-holomorphic discs with one boundary marked point
$
\mathcal M_1 (L, \beta, J)
$.
Since the monotonicity of $L$ means that the symplectic area of discs is positively proportional to their Maslov index, bubbling can be excluded.
Thus, a generic choice of $J$ ensures the regularity of these disks, and then the moduli space $\mathcal M_1 (L, \beta, J)$ is a smooth compact manifold of dimension $n + \mu(\beta)- 2 = n$ where $n=\dim L$.
We can show that the degree of the evaluation map at the marked point $\mathrm{ev}:\mathcal M_1(L,\beta,J)\to L$ is independent of the chosen $J$, and we denote its value by $n_\beta \in\mathbb Z$.

Following Auroux's seminal work \cite{AuTDual}, the counts $n_\beta$ of Maslov index 2 holomorphic disks can be assembled into a \textit{Landau-Ginzburg mirror superpotential}:
\begin{equation}
	\label{eq_W_monotone_intro}
	W_L=\sum_{\mu(\beta)=2} n_\beta  \ e^{-\int_\beta\omega} \ \mathrm{hol}_\nabla (\partial\beta)
\end{equation}
for the context of Strominger-Yau-Zaslow mirror symmetry \cite{SYZ}. Namely, the $W_L$ should be viewed as a global function on certain "mirror space" $X^\vee$ of equivalence classes of Lagrangian submanifolds equipped with unitary flat connections $\nabla$, modulo certain wall-crossing problems  \cite{AuTDual,seidel2008biased} caused by Maslov index 0 holomorphic disks. Here we write $\mathrm{hol}_{\nabla}$ for the holonomy of $\nabla$.
Concerning its symplectic background, we also often call $W_L$ the \textit{disk (super)potential function} associated to $L$.

A simple example of (\ref{eq_W_monotone_intro}) is the monotone toric fiber $L$ in a toric manifold, as in the work of Cho-Oh \cite{Cho_Oh} and Fukaya-Oh-Ohta-Ono \cite{FOOOToricOne}, where the resulting superpotential $W_L$ matches the physical prediction by Hori-Vafa \cite{hori2000mirror}.
More applications of such superpotentials can be found in the works of Kim \cite{kim2023disk},  Nishinou-Nohara-Ueda \cite{nishinou2010toric}, Chanda-Hirschi-Wang \cite{chanda2023infinitely}, Pascaleff-Tonkonog \cite{PT_mutation}, Wu \cite{wu2015exotic}, and many others.

Now, it is natural to raise the question of whether the disk potential function $W_L$ admits a meaningful generalization to non-monotone Lagrangians $L$, and whether such a generalization can yield interesting applications.
The work of Fukaya-Oh-Ohta-Ono \cite[\S 3]{FOOOBookOne} suggests that there exists a potential function
\begin{equation}
	\label{W_FOOO}
W_L:\mathcal M(L)\to \Lambda_0
\end{equation}
where $\mathcal M(L)$ is the moduli \textit{set} of weak bounding cochains and $\Lambda_0$ is the Novikov ring. In general, $\mathcal M(L)$ is only a set and does not carry a natural space structure; see \cite[Theorem 4.3.13]{FOOOBookOne}.

\subsubsection*{\quad Potential function defined on a space}

While the disk potential in \eqref{eq_W_monotone_intro} applies in a more restricted setting than \eqref{W_FOOO}, it has the advantage of being realized as a Laurent polynomial, and hence as a concrete function on a meaningful space. By contrast, the framework of \eqref{W_FOOO} accommodates a much broader class of Lagrangians, but the resulting potential is typically defined only as a function on a set, making it more abstract and less directly suited to applications.
It is therefore natural to ask whether there is an intermediate framework that retains the advantages of both approaches.
This is one of the main motivations for the present work. 
Let's further elaborate on it as follows.

For \eqref{eq_W_monotone_intro}, if we choose a basis $\gamma_1,\dots,\gamma_m$ of $H_1(L)$ and write
$y_i=\mathrm{hol}_\nabla(\gamma_i)\in \mathbb C^*$,
then it can be expressed as a Laurent polynomial function
\[W_L(y_1,\dots,y_m)
=
\sum_{\mu(\beta)=2} n_\beta\,y^{\partial\beta}
:\, (\mathbb C^*)^n \to \mathbb C \]
on the complex torus space,
where $y^{\partial\beta}$ denotes the Laurent monomial corresponding to the boundary class $\partial\beta$. For $\partial\beta=a_1\gamma_1+\cdots+a_m\gamma_m$, we write $y^{\partial\beta}=y_1^{a_1}\cdots y_m^{a_m}$.
Now, let's think of $y_i$ as formal variables $Y_i$, and introduce the Novikov symbol $T$. Then we may recast the above $W_L$ as
\[W_L=\sum_{\mu(\beta)=2} T^{\omega(\beta)} n_\beta Y^{\partial\beta}\]
Here the additional factor $T^{\omega(\beta)}$ is introduced based on \eqref{W_FOOO}. When $L$ is non-monotone, there may exist Maslov index 0 holomorphic disks. If $\alpha\in H_2(X,L)$ can be represented by a holomorphic disk and has Maslov index $\mu(\alpha)=0$, then for any class $\beta\in H_2(X,L)$ with $\mu(\beta)=2$, and any integer $k\ge 0$, we have
$\mu(\beta+k\alpha)=2$,
and these disk classes $\beta+k\alpha$ may all contribute to $W_L$.
Therefore, in general, the potential $W_L$ may contain \textit{infinitely many} monomials, and is typically no longer a Laurent polynomial over $\mathbb C$, but rather a Laurent formal power series over the Novikov field
$\Lambda=\mathbb C((T^{\mathbb R}))$.
In particular, it is natural to invoke transcendental discussions for analytic geometry over $\Lambda$.

Because of convergence issues for infinite series, it does not necessarily define a function $(\Lambda^*)^n\to \Lambda$
directly. However, by estimating the symplectic areas of holomorphic disks, one can prove that on an analytic subspace of the form
\[\trop^{-1}(V)\subseteq (\Lambda^*)^n,\]
where $V\subset \mathbb R^n$ is a small neighborhood containing the origin and where $\trop: (\Lambda^*)^n\to\mathbb R^n$ is the standard tropicalization map (which we will review later in \S \ref{s_trop_map}).
Now, the formal power series $W_L$ does define a well-defined analytic function (in the sense of Berkovich geometry):
\[W_L: \trop^{-1}(V) \to \Lambda.\]

This framework is also compatible with the SYZ mirror symmetry picture suggested around \eqref{eq_W_monotone_intro}. In fact, the so-called notion of affinoid torus fibration may be viewed as a non-archimedean analytic analogue of a Lagrangian torus fibration, locally modeled on the fibration $\trop^{-1}(V)\to V$. An affinoid torus fibration canonically induces an integral affine structure on the base, in analogy with the Arnold-Liouville theorem. In our non-archimedean SYZ mirror construction developed and exemplified in \cite{Yuan_I_FamilyFloer,Yuan_local_SYZ,Yuan_conifold,Yuan_A_n}, these fibrations model the mirror SYZ dual fibration.

Note that $W_L$ generally depends on auxiliary choices. For instance, the existence of Maslov index zero holomorphic disks depends on the almost complex structure $J$. Nevertheless, one can show that up to non‑archimedean analytic isomorphism, $W_L$ is indeed an \textit{invariant} of $L$.
Actually, from the SYZ mirror construction viewpoint, this means that the mirror Landau-Ginzburg mirror superpotential function $W_L$ can be expressed differently in different mirror local charts, where an auxiliary choice basically corresponds to a choice of local chart for the mirror analytic space; we refer to \cite{Yuan_I_FamilyFloer,Yuan_c_1,Yuan_unobs} for more contexts.

%
%
%
%
%

\subsubsection*{\quad Main result}

The main goal here is to derive applications from the well‑definedness of the aforementioned generalized potential function $W_L$, and we consider the following setup.
Suppose $(X, \omega)$ is a symplectic manifold of real dimension $2n$ that is closed or convex at infinity.
Let \(\{L_s:s\in[0,1]\}\) be a smooth Lagrangian isotopy of spin closed graded Lagrangian submanifolds, for instance a Hamiltonian isotopy \(L_s=\phi_s(L_0)\), or a small deformation in a Weinstein neighborhood \(L_s=\operatorname{graph}(\alpha_s)\) for a smooth family of closed one-forms \(\alpha_s\).

Here we further assume that the initial $L_0$ is a graded monotone Lagrangian submanifold.
Recall that the monotone condition means that the area homomorphism $H_2(X,L)\to\mathbb R$ is a positive scalar of the Maslov homomorphism $\mu: H_2(X,L)\to\mathbb Z$.

\begin{thm}
	\label{main_thm}
	If the disk potential \eqref{eq_W_monotone_intro} of $L_0$ is non-vanishing, then $L_s$ must bound a $J$-holomorphic disk of Maslov index 2 for every $s\in [0,1]$ and every $\omega$-tame almost complex structure $J$.
\end{thm}

The point is that the existence of an individual holomorphic disk is not by itself a deformation invariant: moduli spaces may undergo wall-crossing; cf. \cite[Figure 1]{AuTDual}. For example, suppose that $L_0$ bounds several Maslov index two holomorphic disks in classes $\beta_i$. After deforming $L_0$ to $L_1$, the disk predicted by Theorem \ref{main_thm} need not correspond directly to any one of the original classes $\beta_i$. Under a suitable identification of relative homology groups along the isotopy, it may instead be obtained from one of them by adding a Maslov index zero class, namely it may have the class $\beta_i+\alpha$ with $\mu(\alpha)=0$.
Hence, the rough picture is that it is \textit{not} the individual pseudo-holomorphic disks that deform, \textit{but} rather the $A_\infty$ algebras associated to $L_s$, and hence the isomorphism class of the non-archimedean disk potential $W_{L_s}$, that deform.

In this sense, holomorphic disks should not be studied in isolation, but rather organized into a coherent algebraic structure. We therefore propose that the relevant invariant is the homotopy equivalence class of the associated $A_\infty$ algebra. From this perspective, the (generalized) disk superpotential should be viewed functorially: it assigns to each suitable $A_\infty$ algebra a non-archimedean analytic function, and to each equivalence between such algebras the corresponding analytic change of coordinates.

A main motivation of this paper is to demonstrate that the non-archimedean analytic perspective over the Novikov field is useful not only for clarifying the mathematical content of the SYZ conjecture and mirror symmetry, but also for offering new insight into problems in pure symplectic geometry. We hope that the proof of Theorem \(\ref{main_thm}\) provides a tentative step toward the systematic development of this direction. More broadly, we suggest the speculative principle that rigidity phenomena in symplectic geometry may be encoded by suitable non-archimedean analytic structures.

\subsubsection*{\quad Further application}

We now discuss some possible directions for further applications of Theorem \ref{main_thm}. The main input behind Theorem \ref{main_thm} is the construction of a generalized disk potential as a non-archimedean analytic function, together with its well-definedness up to analytic isomorphism. In the present paper, we only use a rather small part of this generalized disk superpotential invariant, namely its non-vanishing. We expect that finer
features of the invariant should lead to further applications.
For instance, one possible direction is to study Newton-type invariants of the generalized potential. In the monotone case, Vianna studies the Newton polytope of the disk superpotential (or the boundary Maslov-2 convex hull) to distinguish Lagrangian tori \cite[Definition 4.2 \& Remark 4.5]{Vianna2016infinitely}. Since the generalized potential constructed here is usually an infinite series rather than a Laurent polynomial, the ordinary Newton polytope should be replaced by a non-archimedean or extended Newton polytope; see e.g. \cite[Section 2.1]{EKL}. Developing this idea concretely would be an interesting direction for future work.

A specific application is obtained by combining Theorem \ref{main_thm} with the Lagrangian isotopy classification of tori.
If all Lagrangian tori in a given symplectic manifold are smoothly
isotopic to a monotone toric fiber with non‑zero potential, then
Theorem \ref{main_thm} implies that they all bound Maslov index two disks for some almost complex structure.
The following is a concrete application of this idea.

\begin{prop}
	\label{Application_4mfd_prop}
	Let $(X,\omega)$ be one of the following symplectic four-manifolds: $\mathbb C^2$, $\mathbb {CP}^2$, $S^2\times S^2$. Then, for every $\omega$-tame almost complex structure $J$, every Lagrangian torus $L\subset X$ bounds a Maslov index two $J$-holomorphic disk.
\end{prop}

To see this, by a theorem of Rizell-Goodman-Ivrii \cite[Theorem A]{dimitroglou2016lagrangian}, any two Lagrangian tori in each of the above symplectic four-manifolds are Lagrangian isotopic. Hence every Lagrangian torus $L\subset X$ is Lagrangian isotopic to a monotone Clifford torus $L_0\subset X$, and its disk potential is nonzero. Indeed, after choosing suitable coordinates, the potentials in the three situations are
$W_{L_0}=Y_1+Y_2$, $W_{L_0}=Y_1+Y_2+\frac{1}{Y_1Y_2}$, and $W_{L_0}=Y_1+\frac{1}{Y_1}+Y_2+\frac{1}{Y_2}$
respectively, up to the Novikov weights determined by symplectic areas. Hence, Proposition \ref{Application_4mfd_prop} follows directly from Theorem \ref{main_thm}.

In 1988, Audin asked whether every Lagrangian torus in $\mathbb C^n$ bounds a holomorphic disk of Maslov index two \cite{audin1988fibres}. The question was first answered in dimension $n=2$ by Viterbo \cite{viterbo1990new} and Polterovich \cite{polterovich1991maslov}. This is basically recovered as a case of Proposition \ref{Application_4mfd_prop}, and we also obtain an Audin-type result for $\mathbb{CP}^2$ and $S^2\times S^2$ from a different approach.

In the monotone setting, the Audin's question was studied by Oh \cite{oh1996floer}, Buhovsky \cite{buhovsky2010maslov}, and Fukaya-Oh-Ohta-Ono \cite[Theorem 6.4.35]{FOOOBookTwo}; see also Damian \cite{damian2012floer}. The general case was proved by Cieliebak-Mohnke \cite{cieliebak2018punctured} using neck-stretching methods from symplectic field theory, and was also proved by Fukaya \cite{fukaya2006application} and Irie \cite{irie2020chain} using $L_\infty$ algebra structures in a loop space framework. More recently, Y. Li obtained certain extension to aspherical Lagrangian submanifolds in Liouville domains with finite first Gutt-Hutchings capacity, using the notion of cyclic dilations \cite{li2023aspherical}.

Note that the proposition is not intended to give another proof of Audin’s conjecture in full generality. Rather, it provides a different mechanism for producing Maslov index two holomorphic disks, and points toward possible connections between this question and other perspectives such as the non-archimedean analytic viewpoint developed in this paper.
 One useful feature of our argument is that the almost complex structure $J$ may be chosen arbitrarily, and it does not rely on the displaceability of $L$ in $X$ as usual. These may suggest that the mechanism could be relevant to disk-existence questions in some more general setting, and we will pursue this direction somewhere else.

\subsection*{Acknowledgment} This project was initiated during the author's visit to the University of Science and Technology of China (USTC) in 2024, when Jun Zhang asked about possible applications of the non-archimedean perspective on the SYZ conjecture. The author is grateful to Jun Zhang and USTC for their hospitality. The author also thanks Kenji Fukaya, Wenyuan Li, Mohammad Rabah, Paul Seidel for helpful conversations.

\section{Aspects of Berkovich geometry}

\subsection{Affinoid algebra and affinoid space}
The \textit{Novikov field} is defined as
\[
\Bbbk=\Lambda= \mathbb C(( T^{\mathbb R})) = \left\{
\sum_{i=0}^\infty a_i T^{\lambda_i} \mid a_i\in\mathbb C, \lambda_i\in\mathbb R,\lambda_i \nearrow +\infty
\right\}
\]
It is known that the Novikov field is a complete algebraically-closed non-archimedean field.
It admits a valuation map, denoted as $
\mathrm{val}: \Lambda\to \mathbb R \cup \{\infty\}$,
is defined by taking the smallest power: If $x=\sum_{i=0}^\infty a_i T^{\lambda_i}$ with $a_0\neq 0$ and $\{\lambda_i\}$ strictly increasing, then we define $\mathrm{val}(x) =\lambda_0$. If $x=0$, then we define $\mathrm{val}(x)=\infty$.
The corresponding norm $|x|=e^{-\mathrm{val}(x)}$ satisfies the non-archimedean triangle inequality $|x+y|\leqslant \max\{|x|, |y|\}$.
The valuation ring is $\Lambda_0:= \mathrm{val}^{-1}[0,\infty]=\{x\mid |x| \leqslant 1 \}$, also called the Novikov ring.
Its maximal ideal is $\Lambda_+:=\mathrm{val}^{-1}(0,\infty]=\{x\mid |x|<1\}$.
The multiplicative group of units is
\[
U_\Lambda:= \mathrm{val}^{-1}(0)=\{x\mid |x|=1\}
\]
Note that $U_\Lambda= \mathbb C^*\oplus \Lambda_+$ and $\Lambda_0 =\mathbb C\oplus \Lambda_+$. The standard isomorphism $\exp: \mathbb C^*\cong \mathbb C/ 2\pi i \mathbb Z$ naturally extends to $U_\Lambda\cong \Lambda_0/ 2\pi i \mathbb Z$. In particular, for any $y\in U_\Lambda$, there exists $x\in \Lambda_0$ with $y=\exp(x)$.

For our purpose, we assume that the ground field $\Bbbk$ is the above Novikov field.
Let $A$ be a commutative $\Bbbk$-algebra with the structure map $\Bbbk\to A$.
A \textit{multiplicative seminorm} on $A$ is a map
$
\|\cdot\|:A\to \mathbb R_{\ge 0}
$, extending the norm on $\Bbbk$, satisfying that for all $f,g\in A$ and all $c\in \Bbbk$, we have: (i) $\|0\|=0$ and $\|1\|=1$; (ii) $\|fg\|=\|f\|\cdot \|g\|$; (iii) $\|f+g\|\le \max\{\|f\|,\|g\|\}$.

Define the \emph{Berkovich spectrum}
$\mathcal M(A)$ to be the set of multiplicative seminorms on $A$ extending the norm on $\Bbbk$.
For $x\in\mathcal M(A)$, we usually denote by $|\cdot|_x: A \to \mathbb R_{\ge 0}$ the corresponding seminorm.
Given $f\in A$, we consider the evaluation map
\[
\ev_f:\mathcal M(A)\to \mathbb R_{\ge 0} \ , \qquad 
x\mapsto |f|_x =: |f(x)|  .
\]
where the notations are standard convention.
The Berkovich topology on $\mathcal M(A)$ is then defined as the coarsest topology for which
$\operatorname{ev}_f$ is continuous for every $f\in A$, and it is known that $\mathcal M(A)$ is a nonempty, compact Hausdorff topological space \cite{Berkovich_2012spectral}.

Fix an integer $n\ge 0$ and a \textit{polyradius} $r=(r_1,\dots,r_n)\in(\mathbb R_{>0})^n$.
We define the \emph{Tate algebra of polyradius $r$} (or \emph{polydisk algebra}) by
\[
\Bbbk \langle r^{-1}T\rangle
:=
\Bbbk \langle r_1^{-1} T_1, \dots, r_n^{-1}T_n \rangle:=
\left\{
\sum_{\nu\in\mathbb Z_{\ge 0}^n} a_\nu T^\nu
\ \middle|\
a_\nu\in \Bbbk,\ \  \lim_{|\nu|\to \infty}  |a_\nu|\, r^\nu = 0
\right\},
\qquad
r^\nu:=\prod_{i=1}^n r_i^{\nu_i},
\]
endowed with the Gauss norm
\[
\sum_\nu a_\nu T^\nu \longmapsto  \max_\nu \bigl(|a_\nu|\, r^\nu\bigr).
\]
When $r=(1,\dots,1)$, $\Bbbk\langle T_1,\dots,T_n\rangle= \Bbbk \langle r^{-1}T\rangle$ is the usual Tate algebra.

A $\Bbbk$-affinoid algebra is a Banach $\Bbbk$-algebra $A$
such that $A \cong \Bbbk\langle r^{-1}T\rangle / I $
for some $r\in(\mathbb R_{>0})^n$ and an ideal $I\subset \Bbbk \langle r^{-1}T\rangle$.
A strictly $\Bbbk$-affinoid algebra is such a quotient with $r=(1,\dots,1)$.
Let's denote by $|\Bbbk^\times|=\{|x| : x \in \Bbbk , x\neq 0\}$ the value group of $\Bbbk$.
When $|\Bbbk^\times|=\mathbb R_{>0}$, as is the case for the Novikov field considered here, then it is known that every $\Bbbk$-affinoid algebra is strictly $\Bbbk$-affinoid.
Accordingly, we will drop the adjective ``strictly'' and also omit the prefix ``$\Bbbk$-'', and just use the term \emph{affinoid algebra} from now on.
The Berkovich spectrum $\mathcal M(A)$ of an affinoid algebra $A$ is called an \textit{affinoid space}.

For affinoid algebras, a bounded homomorphism $\varphi^\sharp:A\to B$
induces a unique morphism of affinoid spaces $\varphi:\mathcal M(B)\to \mathcal M(A)$
by the formula $\varphi(y)=y\circ\varphi^\sharp$.
Conversely, a morphism of affinoid spaces $\varphi$ uniquely determines
the bounded homomorphism $\varphi^\sharp$ on global sections.
Therefore, the category of affinoid spaces is equivalent to the opposite category of
affinoid algebras.

A \emph{locally ringed space} is a pair $(X,\mathcal O_X)$ consisting of a
	topological space $X$ and a sheaf of commutative rings $\mathcal O_X$ on $X$,
	such that for every $x\in X$ the stalk $\mathcal O_{X,x}$ is a local ring.
An \emph{analytic space} is a locally ringed space
	$(X,\mathcal O_X)$ such that every point $x\in X$ admits an open neighborhood $U$
	for which there exists an affinoid $\Bbbk$-algebra $A_U$ and an isomorphism of
	locally ringed spaces
	\[
	(U,\mathcal O_X|_U)\ \cong\ \bigl(\mathcal M(A_U),\,\mathcal O_{\mathcal M(A_U)}\bigr).
	\]
	In other words, $X$ is locally isomorphic (as a locally ringed space) to an affinoid space.

\begin{defn}[Affinoid atlas]
	An \emph{affinoid atlas} on a Berkovich analytic space $(X,\mathcal O_X)$
	is a collection of pairs $\{(U_i,A_i)\}_{i\in I}$ such that:
	
	\begin{enumerate}
		\item the sets $\{U_i\}_{i\in I}$ form an open cover of $X$;
		\item for each $i$ there is an isomorphism of locally ringed spaces
		\[
		(U_i,\mathcal O_X|_{U_i})\ \cong\ \bigl(\mathcal M(A_i),\,\mathcal O_{\mathcal M(A_i)}\bigr).
		\]
	\end{enumerate}
\end{defn}

\begin{rmk}[Compatibility on overlaps]
	Let $\{(U_i,A_i)\}_{i\in I}$ be an affinoid atlas.
	Then for any $i,j$, the overlap $U_i\cap U_j$ is an open subset of both
	$\mathcal M(A_i)$ and $\mathcal M(A_j)$ (via the identifications above).
	The requirement that $(X,\mathcal O_X)$ is a locally ringed space glued from
	affinoid pieces means precisely that the induced identifications on $U_i\cap U_j$
	are compatible with the Berkovich structure sheaves.
	Equivalently, after possibly refining by an affinoid domain cover of $U_i\cap U_j$
	inside $\mathcal M(A_i)$, the transition maps are morphisms of affinoid spaces.
\end{rmk}

Starting from affinoid spaces, Berkovich constructs in \cite{Berkovich1993etale}
the category of \textit{analytic spaces} by a delicate gluing formalism.
Here we give a sketch.
Let $X$ be a topological space and let $(X_i)_{i\in I}$ be a family of subsets of $X$.
We say that $(X_i)$ is a \emph{$G$-covering} of $X$ if every point $x\in X$ admits a neighborhood of the form
$
U=\bigcup_{i\in J} X_i
$
for some \emph{finite} subset $J\subset I$ such that $x\in X_i$ for all $i\in J$.
Notice that $G$-coverings are not required to be open coverings; rather, they are required to cover neighborhoods of points after passing to finitely many members,
each of which contains the point.

Let $X$ be an analytic space in the sense of Berkovich. Then, $X$ is a topological space in which every point admits
a basis of compact neighborhoods.
The space $X$ has a class of compact subsets called its
\emph{affinoid domains}. 
Besides, $X$ also has another class of subsets consisting of the so-called \emph{analytic domains}. By definition, an analytic domain is a subset $V\subset X$ which admits a $G$-covering by
affinoid domains contained in $V$.
The $G$-topology on $X$ is defined as follows: its objects are the analytic domains of $X$, its morphisms are inclusions, and its covering families are the $G$-coverings.
We remark that a compact subset of $X$ which is a finite union of analytic domains of $X$ is an analytic domain.

The spaces most commonly considered in the classical approaches to non-archimedean analytic geometry
(such as Tate's rigid-analytic spaces \cite{Tate_origin})
coincide with what Berkovich calls \emph{good} analytic spaces: those for which every point admits
an affinoid neighborhood, and hence a basis of affinoid neighborhoods.
Good analytic spaces are often sufficient in practice.

To any scheme $X$ of finite type over a non-archimedean field, we can
associate an analytic space $X^{\mathrm{an}}$, called the \emph{analytification} of $X$.
If $X=\mathrm{Spec} A$ is affine, then the underlying set of $X^{\mathrm{an}}$ is the spectrum $\mathcal M(A)$
of multiplicative seminorms $\|\cdot\|_x:A\to\mathbb R_{\ge0}$ extending the norm on the ground field $\Bbbk$.
Alternatively, a point of $X^{\mathrm{an}}$ may be viewed as a pair $(\mathfrak p,|\cdot|_x)$ consisting of a
prime $\mathfrak p\subset A$ together with an absolute value on the residue field $\kappa(\mathfrak p)$ extending that of $\Bbbk$.
For general $X$, choose an affine open cover $X=\bigcup_i U_i$ with $U_i=\mathrm{Spec} A_i$ and set
$U_i^{\mathrm{an}}:=\mathcal M(A_i)$.
On overlaps $U_i\cap U_j$, these analytifications identify with common
analytic domains, and one defines $X^{\mathrm{an}}$ by gluing the spaces $U_i^{\mathrm{an}}$ along these domains.

\subsection{Tropicalization map}
\label{s_trop_map}
Set $\Lambda^* =\Lambda\setminus\{0\}$, and $(\Lambda^*)^n$ can be viewed as the variety $\mathrm{Spec} \ \Lambda[ y_1^\pm,\dots, y_n^\pm]$.
Occasionally, we will simply write $(\Lambda^*)^n$ for its analytification $(\Lambda^*)^{n,\mathrm{an}}$ if the context is clear.
Note that set-theoretically, the space $(\Lambda^*)^{n,\mathrm{an}}$ is the set of multiplicative semi-norms on the Laurent polynomial ring $\Lambda[ y_1^\pm,\dots, y_n^\pm]$ that extends the norm on $\Lambda$.
For example, an $n$-tuple $(y_1,\dots, y_n)$ of points in $\Lambda^*$ gives rise to a semi-norm, sending  a Laurent polynomial $f = \sum_{I} a_I y^I \in \Lambda[y_1^{\pm}, \dots, y_n^{\pm}]$ to $\big| f(y_1, \dots, y_n) \big|$ for the given norm on $\Lambda$. This is a multiplicative seminorm extending the norm on $\Lambda$, hence a point of $(\Lambda^*)^{\,n,\mathrm{an}}$. Such points are often referred to as the \textit{classical points}.

The \textit{tropicalization map} is defined as
\begin{equation}
	\label{trop_eq}
	\mathfrak{trop} : (\Lambda^*)^{n,\mathrm{an}} \to \mathbb R^n
\end{equation}
A modern reference is \cite[\S 12]{Formes_Chambert_2012}.
Explicitly, for a point $x \in (\Lambda^*)^{n,\mathrm{an}}$ corresponding to a multiplicative seminorm, often also denoted by $|\cdot|_x$, on $\Lambda[y_1^{\pm},\dots,y_n^{\pm}]$ extending the norm on $\Lambda$, one uses the coordinate variables $y_i$'s to define
$$\mathfrak{trop}(x) \;=\; \bigl(-\log|y_1|_x,\; \dots,\; -\log|y_n|_x\bigr) \in \mathbb R^n .$$
This is an analogue of the complex logarithmic map $(z_1,\dots,z_n)\mapsto (\log|z_1|,\dots,\log|z_n|)$. It is known that $\mathfrak{trop}$ is a continuous, surjective, and proper map with respect to the Euclidean topology on $\mathbb R^n$ and the non-archimedean analytic topology on $(\Lambda^*)^{n,\mathrm{an}}$.

Let 
\[Z = \mathfrak{trop}^{-1}(0) \subset (\Lambda^*)^{n,\mathrm{an}}
\]
be the central fiber of the tropicalization map, that is, the set of points whose coordinate seminorms are all equal to $1$.
It is the Berkovich spectrum of the Laurent Tate algebra
$$\mathcal{O}(Z) := \Lambda\langle y_1^{\pm},\dots,y_n^{\pm}\rangle,$$
the completion of $\Lambda[y_1^{\pm},\dots,y_n^{\pm}]$ with respect to the Gauss norm $\|\cdot\|_Z$: For $f=\sum_{h\in\mathbb{Z}^n} a_h y^h$, we have
\[
\|f\|_Z=\sup_{z\,\in\,Z}\bigl|f(z)\bigr|  = \max_h\, |a_h|.\]
More specifically, we have
$$\mathcal{O}(Z) = \Bigl\{ \sum_{h\in\mathbb{Z}^n} a_h y^h \;\Big|\; a_h\in\Lambda,\; |a_h|\to 0 \text{ as } |h|\to\infty \Bigr\}.$$

Let $\Delta$ be a rational convex polyhedron in $\mathbb{R}^n$, defined by a finite system of affine-linear inequalities
$$\sum_{j=1}^n b_{ij} u_j \geqslant c_i
\qquad (i = 1,\dots, N)$$
with $c_i\in\mathbb{R}$, $b_{ij}\in\mathbb{Z}$, and $(u_1,\dots, u_n)\in\mathbb{R}^n$.
The preimage of $\Delta$ under the tropicalization map is an affinoid domain in $(\Lambda^*)^{n,\mathrm{an}}$. More precisely, $\mathfrak{trop}^{-1}(\Delta)$ is the Berkovich spectrum of the affinoid algebra
$$
\mathcal A_\Delta= \mathcal O(\trop^{-1}(\Delta) )= \Lambda\langle T^{c_1} y^{-b_1}, \dots, T^{c_N} y^{-b_N} \rangle,$$
where $y^{-b_i} = y_1^{-b_{i1}}\cdots y_n^{-b_{in}}$ and $T$ is the formal variable of the Novikov field. Because the valuation map on $\Lambda$ is surjective onto $\mathbb{R}$, the real constants $c_i$ can be directly used as exponents of $T$ without the need for an abstract pseudo‑uniformizer.
Such domains are called \textit{polyhedral affinoid domains} (cf. the work of Einsiedler, Kapranov, and Lind \cite{EKL}), and these domains form a convenient class of local building blocks for the analytic space structure.

We can view elements in $\mathcal A_\Delta$ as analytic functions on the analytic space $\trop^{-1}(\Delta)$.
We introduce the notation $\|\cdot \|_{\sup,\Delta}$ for the supremum norm (spectral norm) of $\mathcal{A}_\Delta$, given by
\[
\|f\|_{\mathrm{sup}, \Delta}
\;=\; \sup_{x\,\in\,\mathfrak{trop}^{-1}(\Delta)}|f(x)|
\;=\; \max_{h}\;\bigl( |a_{h}| \; e^{-\,\inf_{u\in\Delta}\,\langle u,\,h\rangle} \bigr).
\]
This is the norm that makes $\mathcal{A}_\Delta$ a Banach algebra. 
For any two functions $f,g \in \mathcal{A}_\Delta$ and any point $x$,
we have $|(f+g)(x)| \le \max\bigl(|f(x)|,\, |g(x)|\bigr)$.
Taking the supremum over all $x$ preserves this inequality as
$\sup_x \max (|f(x)|,\, |g(x)| ) = \max (\sup_x |f(x)|,\, \sup_x |g(x)| ).$
Therefore,
$$\|f+g\|_{\sup,\Delta} \le \max\bigl(\|f\|_{\sup,\Delta},\, \|g\|_{\sup,\Delta}\bigr),$$
which is the non‑archimedean (ultrametric) triangle inequality. 

The completion of $\Lambda[y_1^\pm,\dots, y_n^\pm]$ with respect to $\|\cdot\|_{\sup, \Delta}$ is the polyhedral affinoid algebra
\begin{equation}
\label{eq_A_Delta_explicit}
\mathcal{A}_\Delta \;=\; \Bigl\{ \sum_{h\in \mathbb Z^n} a_h y^h \;\Big|\;
\lim_{|h|\to\infty} \bigl( \mathrm{val}(a_h) + \langle u, h\rangle \bigr) = \infty \ \text{for all } u\in\Delta \Bigr\}.
\end{equation}

More generally, for any neighborhood $V$ of $0$, we can set
\[
\mathcal{A}_V \;=\; \Bigl\{ \sum_{h\in \mathbb Z^n} a_h y^h \;\Big|\;
\lim_{|h|\to\infty} \bigl( \mathrm{val}(a_h) + \langle u, h\rangle \bigr) = \infty \ \text{for all } u\in V \Bigr\}.
\]
This need not be an affinoid algebra, and hence need not correspond to an affinoid domain; nevertheless, it still corresponds to the analytic domain \(\trop^{-1}(V)\).
Recall that every affinoid domain is an analytic domain, but not conversely.
For $0\in V $, the condition with $u=0$ forces $\operatorname{val}(a_h)\to\infty$, i.e. $|a_h|\to 0$. Therefore every such series automatically converges on $Z$, giving a natural restriction homomorphism
\begin{equation}
\label{eq_res_Z}
\operatorname{res}_{Z,V}: \mathcal{A}_V \longrightarrow \mathcal{O}(Z).
\end{equation}

\subsection{Germs of analytic functions}
\label{s_germ_function}

Let $\mathcal O$ denote the structure sheaf of analytic functions on $(\Lambda^*)^{n,\mathrm{an}}$ with respect to the analytic topology.
The ring of germs of analytic functions along $Z$ is defined as
$$\mathscr{G} \;:=\; \varinjlim_{\mathcal{U} \supset Z} \mathcal{O}(\mathcal{U}),$$
where the direct limit runs over all analytic open neighborhoods $\mathcal{U}$ of $Z$ in $(\Lambda^*)^{n,\mathrm{an}}$.
Concretely, a germ is represented by a pair $(\mathcal{U}, f)$ with $f \in \mathcal{O}(\mathcal{U})$ an analytic function on $\mathcal{U}$, and two such pairs $(\mathcal{U}, f), (\mathcal{U}', f')$ define the same germ if $f$ and $f'$ coincide on some smaller open set $\mathcal{V} \subseteq \mathcal{U} \cap \mathcal{U}'$ containing $Z$.
Because the preimages $\mathfrak{trop}^{-1}(V)$, with $V$ an open neighbourhood of $0$ in $\mathbb{R}^n$, form a fundamental system of neighbourhoods of $Z$ in the analytic topology, every germ can be represented by a function on a domain of the form $\mathfrak{trop}^{-1}(V)$ where $0 \in V$ is open.
In other words, the germ ring can be also defined as the direct limit
$$\mathscr{G} \;=\; \varinjlim_{V \ni 0} \mathcal{O}\bigl(\mathfrak{trop}^{-1}(V)\bigr) = \varinjlim_{V \ni 0} \mathcal A_V ,$$
where $V$ runs over open neighbourhoods of $0 \in \mathbb{R}^n$.
Concretely, a germ is an equivalence class of pairs $(\mathfrak{trop}^{-1}(V), f)$ with $f\in\mathcal{A}_{V}$, two pairs being equivalent if the functions agree on some common smaller domain. 
A germ is represented by an analytic function $f \in \mathcal{A}_V$ for some $V\ni 0$; two such pairs are identified if the functions coincide on a common smaller domain.

Any element of $\mathscr G$ defines an analytic function on $Z=\trop^{-1}(0)$. Indeed, taking the direct limit over all $V$ in \eqref{eq_res_Z} yields the restriction map \[
\operatorname{res}_Z: \mathscr{G}\to\mathcal{O}(Z).\]

\begin{prop}
\label{prop_res_Z_injective}
$\operatorname{res}_Z$ is injective.
\end{prop}

\begin{proof}
This is essentially the same thing as \cite[Proposition 1.7]{Yuan_unobs}.
Let $f\in\mathscr G$ with $\operatorname{res}_Z(f)=0$. Pick a small neighbourhood $V\ni 0$ and write $f=\sum_{h\in\mathbb Z^n} a_h y^h$ in $\mathcal A_V$. Then $\operatorname{val}(a_h)\to\infty$, so $|a_h|\to 0$ and the set $\{|a_h|\}$ attains a maximum. 
Arguing by contradiction, suppose some $a_h\neq 0$.
By rescaling $f$ if necessary, we may assume $\max_h|a_h|=1$ and pick $h_0$ with $|a_{h_0}|=1$. In particular, all coefficients lie in the valuation ring $\Lambda_0=\{x: |x|\le 1\}$. Reducing modulo $\Lambda_+$ gives a nonzero Laurent polynomial $\bar f=\sum_h \bar a_h y^h$ over $\mathbb C$ with $\bar a_{h_0}\neq 0$ and only finitely many non‑zero terms (since $|a_h|\to 0$). Since $f$ vanishes on $Z$, we obtain that $\bar f(\bar y_1,\dots,\bar y_n)=0$ for all $\bar y\in(\mathbb C^*)^n$. This is impossible for a non‑zero Laurent polynomial. Thus all $a_h=0$ and $f=0$.
\end{proof}

\begin{rmk}
\label{rmk_germ_example}
In general, the map $\operatorname{res}_Z$ is
	not surjective. Let's give an example
	for \(n=1\). Let \(y\) be the coordinate on \((\Lambda^*)^{\mathrm{an}}\).
Consider the series
$f:=\sum_{m=1}^{\infty} T^{\sqrt m}y^m$.
Since $\val (T^{\sqrt m} )=\sqrt m\to\infty$, this defines an analytic function on $Z$.
We claim that \(f\) does not extend to any neighbourhood of \(Z\). Indeed,	suppose that \(f\) extended to an analytic function on
	\(\mathfrak{trop}^{-1}(V)\) for some open neighbourhood \(V\ni 0\).
The series would have to converge on the polyhedral domain corresponding to some interval $[-\delta,\delta]\subset V$. For an analytic function on
	\(\mathfrak{trop}^{-1}([-\delta,\delta])\), the coefficients must satisfy
	$\val(a_m)+um \to +\infty$
	for every \(u\in[-\delta,\delta]\). Applying this to
	\(a_m=T^{\sqrt m}\) and taking \(u=-\delta\), we get
	$\val(a_m)+um =\sqrt m-\delta m$,
	which tends to \(-\infty\), not to \(+\infty\). This is a contradiction.
\end{rmk}

\paragraph{An intrinsic view.}
In the discussion above we used the Laurent polynomial ring
$\Lambda[y_1^{\pm},\dots,y_n^{\pm}]$, which is the group ring
$\Lambda[\mathbb Z^{\,n}]$ of the free abelian group $\mathbb Z^{\,n}$.
Note that we may develop an intrinsic point of view by starting with an arbitrary finitely generated abelian group $\mathcal H$.

By the structure theorem for such groups we can write
$\mathcal H \;\cong\; \mathbb Z^{\,n} \oplus G$ where $G$ is a finite abelian group (the torsion subgroup).
The group ring over the Novikov field is then
$$\Lambda[\mathcal H] \;\cong\; \Lambda[\mathbb Z^{\,n}] \otimes_\Lambda \Lambda[G].$$

Recall that $\Lambda = \mathbb C((T^{\mathbb R}))$ is algebraically closed and of characteristic zero, which may be used to simplify the structure of $\Lambda[G]$.
If $G = \langle g \mid g^m = 1 \rangle \cong \mathbb Z / m \mathbb Z$ is a cyclic group, then choosing the generator $g$ gives an isomorphism of $\Lambda$-algebras
$\Lambda[G] \cong \Lambda[x] / (x^m - 1)$. Since $\Lambda$ is algebraically closed, the polynomial $x^m - 1$ splits completely:
$x^m - 1 \;=\; \prod_{\zeta \in \mu_m} (x - \zeta)$,
where $\mu_m \subset \Lambda$ denotes the set of $m$-th roots of unity.
Since the characteristic of $\Lambda$ is zero, the derivative $m x^{m-1}$ has no common root with $x^m - 1$ and all $m$ roots are distinct. Hence,
$ \Lambda[G]\cong  \Lambda[x] / (x^m - 1) \;\cong\; \prod_{\zeta \in \mu_m} \Lambda[x] / (x - \zeta) \;\cong\; \prod_{\zeta \in \mu_m} \Lambda.$

Let $G$ be an arbitrary finite abelian group. We can decompose it as a direct product of cyclic groups $G \cong \prod_{i=1}^k \mathbb Z / m_i \mathbb Z $.
Using the previous discussion and the fact that group rings turn direct products into tensor products, we have
$\Lambda[G] \;\cong\; \bigotimes_{i=1}^k \Lambda[\mathbb Z / m_i \mathbb Z]
\;\cong\; \bigotimes_{i=1}^k \big( \prod_{\zeta_i \in \mu_{m_i}} \Lambda \big)$.
This is again a direct product of $\Lambda$, that is, $\Lambda[G]$ is of the form $\Lambda\times \cdots \times \Lambda$.
Consequently,
$\Lambda[\mathcal H] \;\cong\; \Lambda[\mathbb Z^{\,n}]\otimes_\Lambda \Lambda[G]$ is a direct product of several $\Lambda[\mathbb Z^n]\cong \Lambda[y_1^\pm, \dots, y_n^\pm]$, and the algebraic variety $\operatorname{Spec}\Lambda[\mathcal H]$ is a
disjoint union of several copies of the algebraic torus $(\Lambda^*)^n$.
All the descriptions remain valid in this setting.

Accordingly, without loss of generality and for simplicity, we shall usually assume that $\mathcal H$ is a finitely generated \textit{free} abelian group. We may then develop the following intrinsic viewpoint.
Let $\mathbf T=\mathbf T_{\mathcal H}$ denote the Berkovich analytification of $\Spec \Lambda[\mathcal H]$. The tropicalization map is then written as
$
\trop=\trop_{\mathcal H}: \mathbf T_{\mathcal H}\to \mathbb R^n,
$
and the central fiber is
$
Z=Z_{\mathcal H}=\trop^{-1}(0),
$
as before; see \eqref{trop_eq}.
One can also similarly define $\mathcal A_\Delta$, $\mathscr G$, $\operatorname{res}_Z$, etc.
By choosing a basis of $\mathcal H$, one can also retrieve those defined in the above.
Anyway, there is no essential difference.

\subsection{Action groupoid for the analytic function germ}
\label{s_groupoid_germ}

An automorphism germ along $Z$ is an equivalence class $[U, \varphi]$ (or simply denoted by $[\varphi]$) of pairs $(U,\varphi)$, where $U$ is an analytic open neighbourhood of $Z$ in $(\Lambda^*)^{n,\mathrm{an}}$, $\varphi : U \to U'$ is an isomorphism of analytic spaces onto another open neighbourhood $U'$ of $Z$, and $\varphi(Z) = Z$.
Two such pairs $(U,\varphi)$ and $(V,\psi)$ are equivalent if they agree on some smaller analytic open neighbourhood $W\subset U\cap V$ of $Z$. The set of equivalence classes forms a group under composition. We denote this group by
$\operatorname{Aut}_Z(\mathscr G)$.

The automorphism group $\operatorname{Aut}_Z(\mathscr G)$ acts on $\mathscr G$ by $[\varphi]\cdot f = f\circ \varphi^{-1}$, which is again an analytic function on a suitable neighbourhood of $Z$, for $[\varphi]\in \operatorname{Aut}_{Z}(\mathscr G)$ and a germ $f\in\mathscr G$.
Let $G \subset \operatorname{Aut}_Z(\mathscr G)$ be a subgroup. We define $\mathscr G /\!/ G$ to be the category whose objects are the elements $f \in \mathscr G$. For $f,g \in \mathscr G$, the set of morphisms from $f$ to $g$ is
$$\operatorname{Hom}_{\mathscr G /\!/ G}(f,g) \;=\; \{\, \Phi \in G \mid \Phi(f) = g \,\}.$$
Composition is given by composition in $G$: if $\Phi : f \to g$ and $\Psi : g \to h$, then $\Psi \circ \Phi : f \to h$.
Then, one can readily check that $\mathscr G /\!/ G$ is a groupoid.
In fact, the identity on an object $f$ is the identity automorphism $\operatorname{id}_{\mathscr G} \in G$. Composition is associative because multiplication in $G$ is associative. Every morphism $\Phi : f \to g$ is invertible: its inverse $\Phi^{-1}\in G$ satisfies $\Phi^{-1}(g)=f$ and therefore yields a morphism $g \to f$.
In particular,
\[
\mathscr G / \!/ \operatorname{Aut}_Z(\mathscr G)\]
is a groupoid.

\section{$A_\infty$ algebras}
\label{s_A_infinity}

We begin with a review of our conventions of curved $A_\infty$ algebras used in \cite{Yuan_unobs,Yuan_I_FamilyFloer}.
Let $A$ be a graded vector space such that given $a\in A$, its grading is denoted by $|a|\in\mathbb Z$.
Assume $d_A: A\to A$ is a differential such that $|d_A(a)|=|a|+1$, and also we write $|d_A|=1$.
More generally, if a multilinear map $\varphi: A^{\otimes k}\to A$ always satisfies $|\varphi(a_1,\dots, a_k)|=p+|a_1|+\cdots+|a_k|$, then we write $|\varphi|=p$.
A basic example we consider is the de Rham complex or cohomology of a manifold $L$, denoted as $\Omega^*(L)$ or $H^*(L)$, but given $\eta$ in $\OL$ or $\HL$, we conventionally define its grading to be 
$|\eta|=\deg\eta -1$ where $\deg \eta$ is the de Rham degree.
This is also known as the \textit{shifted degree}.
The exterior differential also satisfies $|d\eta|=1+|\eta|$.

Suppose $(X,\omega)$ is a symplectic manifold as before and $L$ is a spin closed graded Lagrangian.
The relative homotopy group $H_2(X,L)$ is equipped with the symplectic area $E:H_2(X,L)\to \mathbb R$ and the Maslov index $\mu:H_2(X,L)\to  2 \mathbb Z$.
The unit of $H_2(X,L)$ is denoted as $0$.
For clarity, we will always focus on the following cases of graded vector spaces: $A=\HL$ or $A=\OL$;
$A=\Omega^*(P\times L)$ where $P$ is the convex hull of a finite subset in a Euclidean space.

Given graded vector spaces $A, A'$, we define 
\begin{equation}
	\label{CC_eq}
	\CC_{k,\beta}(A, A')=
	\Hom( A^{\otimes k}, A')
\end{equation}
to be the space of $k$-multilinear operators. Here $\beta\in H_2(X,L)$ is just an extra label but it deforms the grading: we say $\varphi \in \CC_{k,\beta}$ is of homogeneous degree $|\varphi|$ if for any $a_1,\dots, a_k\in A$, we have
\begin{equation}
	\label{grading_deform_eq}
	\mu(\beta)+
	|\varphi(a_1,\dots, a_k)|=|\varphi|+|a_1|+\cdots+|a_k|
\end{equation}
In the special case $k=0$, we view $\CC_{0,\beta}(A, A')=\Hom(\Bbbk, A')\cong A'$ for the ground field $\Bbbk$. For the unit $1$ of $\Bbbk$, $\mu(\beta)+|\varphi(1)|=|\varphi|$, so $|\varphi(1)|=|\varphi|-\mu(\beta)$. However, when the context is clear, we often abuse notation and do not distinguish between $\varphi(1) $ and $ \varphi$.

Define
\begin{equation}
	\label{CC_pi_eq}
	\CC_{\pi}(A, A')  \, \subseteq \prod_{\substack{ k\in\mathbb N \\ \beta\in  H_2(X,L) \\ (k,\beta)\neq (0,0)} } \CC_{k,\beta}(A, A')
\end{equation}
within the direct product, consisting of the systems $\mathfrak t=(\mathfrak t_{k,\beta})_{k\in\mathbb N,\beta\in H_2(X,L)}$ of multilinear operators $\mathfrak t_{k,\beta}$ with the \textit{gappedness conditions}: $\mathfrak t_{0,0}=0$; if $E(\beta)<0$ or $E(\beta)=0$, $\beta\neq 0$, then $\mathfrak t_\beta:=(\mathfrak t_{k,\beta})_{k\in\mathbb N}$ vanishes identically; for any $E_0>0$, there are only finitely many $\beta$ such that $\mathfrak t_{\beta}\neq 0$ and $E(\beta)\le E_0$.
Regarding (\ref{grading_deform_eq}), the degree of an operator system $\mathfrak t$ is denoted as $|\mathfrak t|$.

Next, we introduce
\begin{equation}
	\label{composition_Gerstenhaber_eq}
	\begin{aligned}
		(\g\diamond\f)_{k,\beta}=
		\sum_{\ell \ge 1}
		\sum_{k_1+\dots+k_\ell=k}
		\sum_{\beta_0+ \beta_1+\cdots +\beta_\ell=\beta}
		\g_{\ell,\beta_0} 
		\circ ( \f_{k_1,\beta_1}\otimes \cdots \otimes \f_{k_\ell,\beta_\ell} ) \\
		(\g\{ \h\})_{k,\beta} = \sum_{\lambda+\mu+\nu=k}\sum_{\beta'+\beta''=\beta} \g_{\lambda+\mu+1,\beta'} \circ (\id_{\#\deg'\h}^\lambda \otimes \h_{\nu,\beta''}\otimes \id^\mu)
	\end{aligned}
\end{equation}
Given $\f,\g, \h\in \CC_\pi$, one can check that $\g\diamond \f, \g\{\h\}\in \CC_\pi$ still satisfy the gappedness condition.

\begin{defn}
	\label{A_infty_algebra_defn}
	An \textit{$A_\infty$ algebra with labels} in $H_2(X,L)$ is a pair $(A,\m)$ which consists of an operator system $\m=(\m_{k,\beta})$ contained in $\CC_\pi(A,A)$, satisfying that $|\m|=1$ and the $A_\infty$ associativity relation
	\[\m\{\m\}=0\]
	It is called \textit{minimal} if $\m_{1,0}=0$.
	We call $(A,\m)$ \textit{unital} if there is a \textit{unit} $\one\in A$ such that $|\one|=-1$, $\m_{1,0}(\one)=0$, $\m_{2,0}(\one, x)=(-1)^{|x|-1}\m_{2,0}(x,\one)=x$, and $\m_{k,\beta} (\dots, \one,\dots)=0$ for $(k,\beta)\neq (1,0), (2,0)$. 
	
	An \textit{$A_\infty$ homomorphism with labels in $H_2(X,L)$} from $(A',\m')$ to $(A,\m)$ is defined as an operator system $\f=(\f_{k,\beta})$ in $\CC_\pi(A',A)$ such that $|\f|=0$ and
	\[
	\m\diamond \f=\f\{\m'\}
	\]
	Let $\one_1, \one_2$ be units of $\m'$ and $\m$.
	We call $\f:\m'\to \m$ \textit{unital} if $\f_{1,0}(\one_1)=\one_2$ and $\f_{k,\beta}(\dots, \one_1,\dots)=0$ for $(k,\beta)\neq (1,0)$.
	For simplicity, we will omit mentioning $H_2(X,L)$ if the context is clear.
\end{defn}

\begin{defn}
	$\m\{\m\}=0$ implies that $\m_{1,0}\circ \m_{1,0}=0$, and $\f_{1,0}$ is a cochain map from $(A',\m'_{1,0})$ and $(A,\m_{1,0})$.
	If $\f_{1,0}$ is a quasi-isomorphism, $\f$ is called \textit{an $A_\infty$ homotopy equivalence} or \textit{$A_\infty$ quasi-isomorphism}.
\end{defn}


\subsection{Algebra framework}
Let $P$ be the convex hull of a finite subset in a Euclidean space.
The pullback maps for the inclusion maps $L\to \{s\}\times L\subset P \times L$ and the projection map $P\times L\to L$ give rise to
$\eval^s: \Omega^*(P\times L)\to \Omega^*(L)$
for every $s\in P$ and 
$
\incl: \Omega^*(L) \to \Omega^*(P\times L)
$.

An element in $\Omega^*(P\times L)$ can be written in the form $\sum_{i=1}^n \alpha_i \otimes x_i(s)$, where $\alpha_i\in \Omega^*(P)$ and $x_i(s)\in \Omega^*(L)$ smoothly varies with $s\in P$.
Namely, it can be viewed as an $\Omega^*(L)$-valued differential form on $P$, and we have
\[
\Omega^*(L)_P :=\Omega^*(P)\otimes_{C^\infty(P)} C^\infty(P, \Omega^*(L)) \cong \Omega^*(P\times L)
\]
In general, we may introduce 
\[
A_P:= \Omega^*(P) \otimes_{C^\infty(P)} C^\infty(P, A)
\]
whenever $C^\infty (P, A)$ can be defined, such as when $A$ is a finite-dimensional vector space like $\HL$.
Assume $d=d_A$ is a fixed differential on $A$ such as the zero differential $d=0$ on $\HL$ or the exterior differential $d$ on $\OL$. Then there is a corresponding differential $d_P$ on $A_P$ defined by sending $\alpha\otimes x(s)$ to $\alpha\otimes d(x(s))+ \sum (-1)^{|\alpha|} ds_k\wedge \alpha \otimes \partial_{s_k} x(s)$ where $(s_k)$ are coordinates for $P$.
For simplicity, we write $d_{P}=1\otimes d+\sum ds_k\otimes \partial_{s_k}$.

\begin{defn}
	\label{pointwise-defn}
	A \textit{$P$-pseudo-isotopy of $A_\infty$ algebras} on $A$ is defined as an $A_\infty$ algebra $\M=(\M_{k,\beta})$  within $\CC_\pi(A_P, A_P)$ additionally satisfying that $\M_{1,0}=d_{P}$ as above and every $\M_{k,\beta}$ is $P$-pointwise in the following sense:
	\begin{align*}
		&\M_{k,\beta}(\eta_1 \otimes x_1,\dots, \sigma\wedge \eta_i\otimes x_i,\dots, \eta_k \otimes x_k ) \\
		&= (-1)^{\deg \sigma \cdot \left( |\M|
			+\sum_{a=1}^{i-1} (\deg \eta_a+|x_a|) \right)} \sigma \wedge \M_{k,\beta}
		(\eta_1\otimes x_1,\dots, \eta_k\otimes x_k)
	\end{align*}
	where $\deg$ is the usual de Rham degree and $|\cdot|$ is the shifted degree.	 
	If $P=\oi$, we often omit $P$.
	When we say a $P$-pseudo-isotopy is unital if it has a unit in the form $\incl(\one)$ for some $\one\in A$.
\end{defn}

The $P$-pointwise condition allows us to express the $P$-pseudo-isotopy $\M$ as 
\begin{equation}
	\label{1otimes_s_eq}
	\M=1\otimes \m^s+ \sum_{\varnothing \neq I \subseteq \{1,2,\dots, m\}} ds_I\otimes \mc^{I,s}
\end{equation}
where $\m^s$ defines an $A_\infty$ algebra and $ds_I=ds_{i_1}\wedge \cdots \wedge ds_{i_k}$ for $I=\{i_1,\dots, i_k\}$.
Each $(A, \m^s)$ is an $A_\infty$ algebra, called the \textit{restriction} of $(A_P, \M)$ at $s\in P$.
The following result is due to \cite{FuCyclic}. See also \cite{Yuan_I_FamilyFloer,Yuan_unobs} and \cite[21.25 \& 21.29]{FOOO_Kuranishi}.

\begin{lem}
	\label{pseudo_isotopy_pointwise_lem}
	When $P=\oi$, write the pseudo-isotopy $\M=1\otimes \m^s+ ds\otimes \mc^s$ as above. The $A_\infty$ relation $\M\{\M\}=0$ is equivalent to the following conditions:
	(i) $(A,\m^s)$ is an $A_\infty$ algebra, that is, $\m^s\{\m^s\}=0$;
	(ii) $\m^s_{1,0}=d$, independent of $s$, and $\mc^s_{1,0}=\frac{d}{ds}$; (iii) one has
	\[
	\hspace{-2em}
	\frac{d}{ds}\m^s_{k,\beta} 
	+ 
	\sum_{\substack{i+j+\ell=k }} 
	\sum_{\substack{ \beta_1+\beta_2=\beta\\
			(i+j+1,\beta_1)\neq (1,0)}}
	\mc^s_{i+j+1,\beta_1}  (\id_\#^i \otimes \m^s_{\ell,\beta_2} \otimes \id^j) 
	\ \ -
	\sum_{\substack{i+j+\ell=k }} 
	\sum_{\substack{ \beta_1+\beta_2=\beta \\
			(\ell,\beta_2)\neq(1,0)}}
	\m^s_{i+j+1, \beta_1} 
	(\id^i \otimes 
	\mc^s_{\ell,\beta_2} \otimes \id^j) =0
	\]
	Indeed, the last condition means 
	$
	\mc^s\{\m^s\}-\m^s\{\mc^s\}=0
	$.
	Alternatively, for the reduced form $\tilde \mc^s= \mc^s-\mc_{1,0}^s$ removing the $(1,0)$-component, we have
	\[
	\tfrac{d}{ds} \m^s + \tilde \mc^s \{\m^s\} -\m^s\{\tilde \mc^s\} =0
	\]
\end{lem}

Note that we have a natural group homomorphism
\begin{equation}
	\label{partial_pi_2_X_L_eq}
	\partial: H_2(X,L) \to    H_1(L)
\end{equation}
Let $Z(A)=\{a\in A\mid d a=0, \ |a|=0 \}$.
In practice, we basically only concern the cases $A=\HL_P$ or $\OL_P$ for which $|a|=0$ means $\deg a=1$.
For $b\in Z(A)$ and $\sigma\in H_1 (L)$, we define $\sigma\cap b$ as follows.
First, we may write 
$
b=1\otimes \underline b(s)+\sum_i ds_i\otimes  b_i(s)
$
where $s=(s_i)$ are coordinates of $P$ and $\underline b(s), b_i(s)\in \HL$ or $\OL$.
If $A=\HL_P$, then $b\in Z(A)$ implies that $\partial_{s_i} \underline b(s)=0$, that is, $\underline b(s)\in H^*(L)$ is independent of $s$. Hence, we can define $\sigma\cap b:=\sigma \cap \underline b(s)$.
If $A=\OL_P$, then $d(\underline b(s))=0$ and $\partial_{s_k} \underline b(s) -d_L(b_k(s))=0$. Thus, the de Rham cohomology class $\mathfrak b$ of $\underline b(s)$ is independent of $s$, so we may define 
$
\sigma \cap  b :=\sigma \cap \mathfrak b
$.

\begin{defn}
	An operator system $\mathfrak t = (\mathfrak t_{k,\beta})$ in $\CC_\pi (A, A')$ is said to satisfy the \textit{divisor axiom} if for any $b\in Z(A)$ and $(k,\beta)\neq (0,0)$, we have
	$
	\sum_{i=1}^{k+1}  \mathfrak t_{k+1,\beta}(x_1,\dots, x_{i-1}, b, x_i,\dots ,x_k) =
	\partial \beta \cap b \cdot \mathfrak t_{k,\beta}(x_1,\dots,x_k)
	$
	An operator system $\mathfrak t=(\mathfrak t_{k,\beta})$ is said to be \textit{cyclically unital} if, for any $\e\in A$ with $|\e|=-1$ and $(k,\beta)\neq (0,0)$, we have
	$	\sum_{i=1}^{k+1} (-1)^{\sum_{a=1}^{i-1}|x_a|} \mathfrak t_{k+1,\beta} (x_1,\dots, x_{i-1}, \e, x_i, \dots, x_k) =0$
\end{defn}

\begin{prop}
	\label{UD_prop}
	\emph{There is a category $\UD\equiv \UD(L)$ with the following:}
	\begin{enumerate}
		\setlength{\itemsep}{0em}
		\item[(I)] \emph{An object in $\UD$ is an $A_\infty$ algebra $(A,\m)$ satisfying: (I-1) it is a $P$-pseudo-isotopy with $A=\HL_P$ or $\OL_P$, and $\m_{1,0}$ is equal to the natural differential; (I-2) the constant-one function gives the unit; (I-3) it is cyclically unital; (I-4) it satisfies the divisor axiom; (I-5) if $\m_{k,\beta}\neq 0$, then $\mu(\beta)\ge 0$.}
		
		\item[(II)] \emph{A morphism $\f$ in ${\UD}$ is an $A_\infty$ homomorphism such that
			(II-1) it is unital with respect to the constant-one units; (II-2) it is cyclically unital; (II-3) it satisfies the divisor axiom; (II-4) for $b\in Z(A)$, we have $\partial \beta \cap \f_{1,0} (b)   =  \partial \beta \cap b $; (II-5) if $\f_{k,\beta}\neq 0$, then $\mu(\beta)\ge 0$.}
	\end{enumerate}
\end{prop}

\begin{proof}
	This means the listed properties are closed under $A_\infty$-homomorphism composition. All verification should be  algebraic and straightforward; see \cite{Yuan_I_FamilyFloer,Yuan_unobs} for more details.
\end{proof}


We will refer to the conditions (I-5) and (II-5) as the \textit{semipositive conditions}. Geometrically, the $\mu$ stands for the Maslov index, and the condition is valid for the working $A_\infty$ algebras associated to graded Lagrangian submanifolds \cite[Lemma 3.1]{AuTDual}.
\textit{Henceforth, all $A_\infty$ algebras and $A_\infty$ homomorphisms should belong to the category $\UD$, unless explicitly stated otherwise.}

The {ud-homotopy} among morphisms in $\UD$ is introduced in \cite{Yuan_I_FamilyFloer}.

\begin{defn}
	\label{defn_ud_homotopy}
	We write 
	$
	\f_0\simud\f_1
	$ and
	say two morphisms $\f_0,\f_1$ from $(A',\m')$ to $(A,\m)$ in $\UD$ are \textit{ud-homotopic}, 
	if there is a morphism $\F$ from $(A',\m')$ to $(A_\oi, \M^{\tri})$ in $\UD$ such that $\eval^0 \F=\f_0$ and $\eval^1\F=\f_1$, where $\M^\tri$ is the trivial pseudo-isotopy induced by $\m$.
	Equivalently, it means there exists $\f_s,\h_s$ in $\CC_\pi(A',A)$, varying smoothly with $s\in\oi$, such that
	
	\begin{itemize}
		\setlength{\itemsep}{0.3em}
		\item[(a)] Each $\f_s$ is a morphism in $\UD$ from $(A',\m')$ to $(A,\m)$;
		\item[(b)] $ \frac{d}{ds} \circ \f_s = \h_s\{\m\} + \sum (-1)^\ast \ \m'\circ (\f_s \otimes \cdots \otimes \f_s \otimes  \h_s\otimes \f_s\otimes \cdots\otimes \f_s)$;
		\item[(c)] The $\h_s$ satisfies the divisor axiom, the cyclical unitality, and $(\h_s)_{k,\beta}(\cdots \one \cdots )=0$ for all $(k,\beta)$;
		\item[(d)] $|(\h_s)_{k,\beta}|= -1$. For $\beta$ with $\h_\beta\neq 0$, we have $\mu(\beta)\ge 0$.
	\end{itemize}
\end{defn}

The following result is proved in \cite{Yuan_unobs,Yuan_I_FamilyFloer}; a slightly simpler case is also known in \cite[Theorem 4.2.45]{FOOOBookOne}

\begin{prop}[Whitehead theorem]
	\label{prop_whitehead}
	Fix $\f\in \Hom_\UD ((A',\m'), (A,\m))$ such that $\f_{1,0}$ is a quasi-isomorphism of cochain complexes. Then, there exists $\g\in \Hom_\UD ( (A,\m) , (A',\m'))$, unique up to ud-homotopy, such that $\g\diamond \f\simud \id_{A'}$ and $\f\diamond \g\simud \id_A$. We call $\g$ a \emph{ud-homotopy inverse} of $\f$.
\end{prop}



Next, we review a variant of homological perturbation.

Every $A_\infty$ algebra $\check \m$ is also accompanied by a minimal $A_\infty$ algebra $\m$, called the \textit{minimal model}, provided a choice of the so-called contraction.
Let $A$ and $H$ be two graded vector spaces equipped with differentials $d$ and $\delta$ respectively.
A triple $(i,\pi,G)$, consisting of two maps $i: H \to A$, $\pi: A \to  H$ of degree $0$ and a map $G: A \to A$ of degree $-1$, is called a \textit{contraction} if $d \circ i = i \circ \delta $, $\pi \circ d = \delta \circ \pi $, and $i\circ \pi-\id_{C} = d \circ G + G\circ d$.
Further, the $(i,\pi,G)$ is called a \textit{strong contraction}, or say it is \textit{strong}, if we have the extra conditions
$\pi\circ i - \id_{\mH}=0$, $G\circ G=0$, $G\circ i =0$, and $\pi\circ G=0$.
We can extend the standard homological perturbation process (cf. \cite[Theorem 5.4.2]{FOOOBookOne}) as follows.

\begin{prop} [Homological perturbation]
	\label{prop_minimal_model_homological_perturbation}
	Fix an $A_\infty$ algebra $(A,\check \m)$ with labels and a graded cochain complex $(H,\delta)$.
	Suppose $\check\m_{1,0}=d$ and $g=(i,\pi, G)$ is a contraction.
	Then, there is a canonical way to construct an $A_\infty$ algebra $(H, \m)$ and an $A_\infty$ homotopy equivalence
	$
	\mi: (H,\m)\to (A,\check  \m)
	$
	such that 
	$
	\mi_{1,0}=i$ and $\m_{1,0}=\delta$. We call $(H,\m)$ the \textit{minimal model $A_\infty$ algebra} of $(A,\check \m)$ with respect to $g$.
\end{prop}

\begin{proof}[Sketch of Proof]
	The induction formulas for them are 
	\begin{equation}
		\label{eq_induction_tree_minimal_model}
		\begin{aligned}
			\mi_{k,\beta} &= 
			\sum_{k_1+\cdots+k_\ell=k} \sum_{\beta_1+\cdots+\beta_\ell=\beta} \sum_{ (\ell,\beta_0)\neq (1,0)}
			G \circ \check \m_{\ell,\beta_0} \circ (\mi_{k_1,\beta_1}\otimes \cdots \otimes \mi_{k_\ell,\beta_\ell}) \\
			\m_{k,\beta} &= 
			\sum_{k_1+\cdots+k_\ell=k} \sum_{\beta_1+\cdots+\beta_\ell=\beta}
			\sum_{ (\ell,\beta_0)\neq (1,0)}
			\pi \circ \check  \m_{\ell,\beta_0} \circ (\mi_{k_1,\beta_1}\otimes \cdots \otimes \mi_{k_\ell,\beta_\ell})
		\end{aligned}
	\end{equation}
	We can check that they satisfy the desired $A_\infty$ relations.
	We refer to \cite[Theorem 2.52]{Yuan_I_FamilyFloer} for more details. Remark that in the special case $\delta=0$, it is also established in \cite{FuCyclic}.
\end{proof}

In practice, we mainly study the following two specific contractions:

\begin{itemize}
\itemsep 2pt
	\item 
	Take two graded cochain complexes $(\Omega^*(L), d)$ and $(H^*(L), d=0)$.
	Let $g$ be a Riemannian metric on $L$. The Hodge decomposition expresses $\Omega^*(L)$ as the direct sum of the space $\mathcal H_g$ of harmonic differential forms, and the images of the exterior derivative $d$ and its adjoint $d^*$. As $\mathcal H_g$ is naturally identified with $H^*(L)$, we can define $i = i(g)$ as the inclusion map from $\mathcal H_g$ to $\Omega^*(L)$, $\pi = \pi(g)$ as the projection from $\Omega^*(L)$ to $\mathcal H_g$.  It is standard that there exists a natural operator $G=G(g)$ such that $(i,\pi,G)$ is a strong contraction.
	Following \cite{Yuan_I_FamilyFloer}, we call it the \textit{harmonic contraction} induced by $g$ or the \textit{$g$-harmonic contraction}.
	
	\item Consider two graded cochain complexes $(\Omega^*(L)_\oi, d_{L\times \oi})$ and $(\HL_\oi, d=d_{[0,1]})$.
	Given a family of metrics $\pmb g=(g_s)_{0\leqslant s \leqslant 1}$ on $L$, let $(i_s,\pi_s,G_s)=(i(g_s),\pi(g_s), G(g_s))$ be as above. We can find operators $h_s$, $k_s$, and $\sigma_s$ such that $i(\pmb g)= 1\otimes i_s +ds\otimes h_s$, $\pi(\pmb g)= 1\otimes \pi_s +ds\otimes k_s$, and $G(\pmb g)=1\otimes G_s+ds\otimes \sigma_s$ form a strong contraction, called \textit{$\pmb g$-harmonic contraction}. See \cite[Lemma 3.6]{Yuan_I_FamilyFloer}.
\end{itemize}

%
%
%

\begin{prop}
	\label{prop_minimal_model_UD}
	Given an object $(\OL, \check \m)$ in the category $\UD$, utilizing Proposition \ref{prop_minimal_model_homological_perturbation} with the strong contraction $g$ transforms $\check \m$ into another object $(\HL, \m)$ in $\UD$, and simultaneously yields $\mi:\m\to\check \m$, an $A_\infty$ homotopy equivalence in $\UD$.
	Moreover, for a pseudo-isotopy $(\OL_\oi, \check \M)$ in $\UD$, applying Proposition \ref{prop_minimal_model_homological_perturbation} to $\check \M$ with the $\pmb g$-harmonic contraction produces a pseudo-isotopy $(\HL_\oi, \M)$ in $\UD$, and simultaneously yields an $A_\infty$ homotopy equivalence $\mI: \M \to \check \M$ in $\UD$.
\end{prop}

The above proposition is actually straightforward to check from the definitions.
	See \cite[Proposition 2.55 and Theorem 3.3]{Yuan_I_FamilyFloer}, and see \cite[Lemma 3.8 and Theorem 3.9]{Yuan_I_FamilyFloer} for more details.


\begin{prop}
	\label{from_pseudo_to_A_homo_prop}
	For $a<b$, there is a canonical way to assign to a pseudo-isotopy $( A_\ab, \M)$ an $A_\infty$ homomorphism 
	$
	\mC=\mC^\ab: (A, \m^a ) \to (A, \m^b)
	$
	with $\mC_{1,0}=\id_A$.
	Moreover, given a pseudo-isotopy $(A_\ab, \M)$ which is an object in $\UD$, its restriction $A_\infty$ algebra $\m^s$ at each $s\in \oi$ is also an object in $\UD$, and $\mC^\ab$ is a morphism in $\UD$.
\end{prop}

\begin{proof}[Sketch of proof]
	If we write $\M=1\otimes \m^s+ds\otimes \mc^s$ as before, then the following inductive formula is
	\begin{equation}
		\label{inductive_formu-Fubini-yield-eq}
		\mC_{k,\beta}^\ab =
		\sum_{\ell\ge 1}
		\sum_{ \substack{ 
				\beta_0+\beta_1+\cdots+ \beta_\ell = \beta\\
				k_1+\cdots+k_\ell = k \\
				(\ell,\beta_0)\neq(1,0)}
		}
		-\int_a^b du \cdot 
		\mc^u_{\ell,\beta_0}
		\circ
		\big(
		\mC_{k_1,\beta_1}^{[a,u]} \otimes \cdots \otimes \mC_{k_\ell,\beta_\ell}^{[a,u]}
		\big)
	\end{equation}
	for $(k,\beta)\neq(0,0),(1,0)$.
	Using the inductive formula, the desired conditions are direct to check; see \cite{FuCyclic} and \cite[Theorem 2.56]{Yuan_I_FamilyFloer}.
\end{proof}

Note that the condition (\ref{inductive_formu-Fubini-yield-eq}) can be also concisely viewed as
\begin{equation}
	\label{du_C_eq}
	\tfrac{d}{du} \mC^{[a,u]} = -\tilde\mc^u \diamond \mC^{[a,u]} \qquad (i=0,1)
\end{equation}
Here $\tilde \mc^u=\mc^u-\mc^u_{1,0}$ is the reduced form of $\mc^u$, removing the $\CC_{1,0}$-component of $\mc^u$, as in Lemma \ref{pseudo_isotopy_pointwise_lem}.

%
%

\subsection{Axiomatization of Fukaya's $A_\infty$ algebras}
\label{s_axiomatization}
We review the $A_\infty$ algebras for Lagrangians.
Our approach is based on the latest published monograph \cite{FOOO_Kuranishi} of Fukaya-Oh-Ohta-Ono, regarding the de Rham model as opposed to their earlier singular chain model in \cite{FOOOBookOne,FOOOBookTwo}. This de Rham model is also studied by Solomon and Tukachinsky \cite{solomon2016differential} \cite[Section 2]{solomon2016point}.

Let $\mathfrak J(X)$ be the space of almost complex structures tamed by $\omega$.
Let $L$ be a spin closed graded Lagrangian submanifold in $X$.
Given $J\in\mathfrak J(X)$, $\beta\in H_2(X,L)$, and $k\in\mathbb N$ with $(k,\beta)\neq (0,0), (1,0)$, we denote by 
$
\mathcal M_{k+1,\beta}(J,L)
$
the moduli space of equivalence classes $[\Sigma, u, \mathbf z]$ of \textit{$(k+1)$-boundary-marked $J$-holomorphic stable maps of genus 0 with one disk component bounded by $L$ in the class $\beta$}.
See \cite{MS,Frau08,FOOOBookOne}.
For simplicity, we sometimes just call them \textit{stable disks} if the context is clear.
Here the $\Sigma$ is a nodal surface with one boundary component $\partial \Sigma$, the $u: (\Sigma, \partial\Sigma)\to (X,L)$ is continuous and $J$-holomorphic restricted on each irreducible component, the $\mathbf z=(z_0,z_1,\dots,z_k)$ is the boundary marked points ordered counter-clockwisely. 
We also require $u$ is stable; 
The above equivalence relation refers to biholomorphic maps on the domains of two stable maps which identifies the nodal points, the marked points and the boundaries. 

The moduli space $\mathcal M_{k+1,\beta}(J,L)$ is first a well-defined \textit{set} and is also a compact Hausdorff \textit{topoloigcal space}, equipped with the stable map topology \cite[Theorem 7.1.43]{FOOOBookTwo}.
Abusing the notations, let $ J=\{J_t\mid t\in P\}$ be a smooth family of $\omega$-tame almost complex structures parameterized by a convex set $P$, we write
\begin{equation}
	\label{moduli_space_P}
	\textstyle
	\mathcal M_{k+1,\beta}(  J,L):=\bigsqcup_{t\in P} \ \{t\} \times \mathcal M_{k+1,\beta}(J_t,L)
\end{equation}
and the \textit{evaluation map} at the $i$-th marked point $z_i$ is
$
\ev_i:\mathcal M_{k+1,\beta}(J,L) \to L$ where $ i=0,1,2,\dots, k
$.
For the exceptional case $\beta=0$ in $H_2(X,L)$, we set $\mathcal M_{k+1,0}(J,L)=L\times \mathbb D^{k-2}$ or $P\times L\times \mathbb D^{k-2}$ if $k\geqslant 2$. We do \textit{not} define $\mathcal M_{k+1,0}(J,L)$ if $k=0,1$. See \cite[21.6(V)]{FOOO_Kuranishi}.

We call the collection 
\begin{equation}
	\label{moduli_space_system_eq}
	\mathbb M(J,L) =\Big\{ \mathcal M_{k+1,\beta}(J,  L) \mid (k,\beta)\in \mathbb N\times  H_2(X,L)	\Big\}
\end{equation}
the \textit{moduli space system} associated with $(J,L)$.

\begin{thm}
	\label{axiom_J_thm}
	We can associate to the moduli space system $\mathbb M(J,L)$ an $A_\infty$ algebra on $\Omega^*(L)$ independent of the choices up to pseudo-isotopy. Specifically,
	\begin{enumerate}[(a)]
		\itemsep 2pt
		\item \emph{Fix $J\in \mathfrak J(X)$ and $\mathbb M(J,L)$. There exists a set $\mathbb V(J; L)$ of virtual fundamental chains in which any element $\Xi$ is associated with an $A_\infty$ algebra $\check \m^{J,\Xi}$.}
		
		\item \emph{Fix $J, J'\in\mathfrak J(X)$ and $\Xi\in \mathbb V(J; L)$, $\Xi'\in\mathbb V(J' ; L)$ defining the $A_\infty$ algebras $\check \m^{J,\Xi}$, $\check \m^{J',\Xi'}$ as above.
			Take a path $\pmb J=(J_t)_{0\leqslant t \leqslant 1}$ connecting $J$ and $J'$ and the moduli system $\mathbb M(\pmb J,L)$, there exists a set $\mathbb V(\pmb J, \Xi, \Xi'; L)$ of virtual fundamental chains in which any element $\pmb \Xi$ is associated with a pseudo-isotopy $\check \M^{\pmb J , \pmb \Xi}$ on $\Omega^*(L)_\oi$ such that it restricts to $\check \m^{J,\Xi}$ and $\check \m^{J',\Xi'}$ at the two ends.}
	\end{enumerate}
\end{thm}

A reader uncomfortable with virtual techniques may think of Theorem \ref{axiom_J_thm} as the \textit{axioms} for Fukaya's $A_\infty$ algebra:
To establish this, we refer to the appendix of \cite{Yuan_unobs} for an exposition based on the previous works of Fukaya et al \cite[Theorem 21.35]{FOOO_Kuranishi} (see also \cite{FOOODiskOne,FOOODiskTwo,FuCyclic}) to achieve these structures.
Note that $J$ represents concrete geometric data for almost complex structures, while $\Xi$ is an abstract choice of data for virtual techniques.

For our purposes, we shall also impose the following assumption.
 
\begin{assumption}
\label{assumption_empty_moduli}
In the setting of Theorem \(\ref{axiom_J_thm}\), for any auxiliary choice $\Xi$, the associated $A_\infty$ algebra \(\check{\m}^{J,\Xi}\) on $\Omega^*(L)$ satisfies the following vanishing property: if for $\beta \in H_2(X,L)$ with $\beta\neq 0$, there is no nonconstant $J$-holomorphic disk in the class $\beta$, then
$\check{\m}^{J,\Xi}_{k,\beta}=0$
for all $k\in\mathbb N$.
Moreover, for a one-parameter family $\pmb J=(J_t)_{t\in[0,1]}$ and any auxiliary choice $\pmb\Xi$, the associated pseudo-isotopy \(\check{\M}^{\pmb J,\pmb\Xi}\) satisfies the analogous property: if nonzero $\beta\in H_2(X,L)$ admits no nonconstant $J_t$-holomorphic disk for any $t\in[0,1]$, then
$\check{\M}^{\pmb J,\pmb\Xi}_{k,\beta}=0$
for all $k\in\mathbb N$.
\end{assumption}

\begin{rmk}
	\label{rmk_m_10_20}
For $\check \m=\check \m^{J,\Xi}$ as above, we have that $\check \m_{1,0}=d$ is the de Rham differential and $\check \m_{2,0}$ is the (signed) wedge product; see \cite[Deﬁnition 21.21 (2) \& (3)]{FOOO_Kuranishi}.
	The similar holds for the item (b); see \cite[Deﬁnition 21.29 (4) \& (5)]{FOOO_Kuranishi}.
\end{rmk}

The geometric intuition behind this assumption should be quite clear. The operation \(\check{\m}_{k,\beta}^{J,\Xi}\) is defined geometrically using the moduli space $\mathcal M_{k+1,\beta}(J,L)$. Hence, when this moduli space is empty, the corresponding operation, as defined for instance in \cite[Eq (22.17)]{FOOO_Kuranishi}, is expected to vanish. Similarly, one expects other virtual techniques, such as polyfold theory \cite{PolyfoldI}, to yield the same vanishing property, since the underlying moduli spaces are the same, though the virtual perturbation theories may differ. A complete treatment of this would take us beyond the scope of the present paper; therefore, we formulate it as the above assumption.
We also remark that the converse of the above assumption need not hold: the vanishing of the corresponding $A_\infty$ operations does not imply that the relevant moduli spaces are empty. Indeed, the moduli spaces may be nonempty, while the contributions of their elements cancel in the resulting operation.

\section{Superpotential as a functor from $A_\infty$ algebras to analytic functions}

Recall that $L$ is a closed spin graded Lagrangian submanifold of $X$. We denote by $\Lambda[[H_1(L)]]$ the \emph{formal group ring} of $H_1(L)$ over $\Lambda$: its elements are formal series of the form
$$\sum_{k=1}^{\infty} c_k Y^{\alpha_k},$$
where $c_k\in \Lambda$ and $\alpha_k\in H_1(L)$, with multiplication determined by $Y^\alpha Y^\beta = Y^{\alpha+\beta}$.
For simplicity, let us always assume that $H_1(L)$ is torsion-free; as we discussed at the end of Section~\ref{s_germ_function}, one can readily see that this is without loss of generality eventually. Note that specifying a basis $\{e_1,\dots, e_m\}$ of $H_1(L)$ yields an isomorphism $\Lambda[[H_1(L)]]\cong \Lambda[[Y_1^{\pm 1},\dots, Y_m^{\pm 1}]]$ via $Y^{e_i}\leftrightarrow{} Y_i$.

Let $\mathbf T_L$ denote the analytification of $\Spec \Lambda[H_1(L)]$. After choosing a basis of $H_1(L)$, we may identify $\mathbf T_L$ with $(\Lambda^*)^{n,\mathrm{an}}$ as in \eqref{trop_eq}. We then define the tropicalization map
$
\trop:\mathbf T_L\to \mathbb R^n$
and the central fiber $Z$ in the same way as in Section \ref{s_trop_map}.
Let $\mathscr G=\mathscr G_L$ be the germ of non-archimedean analytic functions along $Z$ as defined in Section \ref{s_germ_function}.

\subsection{Associated series}

We focus on those $A_\infty$ structures with underlying graded vector space fixed to be $H^*(L)$. 
Let \((H^*(L),\m)\) be an $A_\infty$ algebra in \(\UD=\UD(L)\). By \eqref{grading_deform_eq}, the degree condition \(|\m_{0,\beta}|=|\m|=1\) implies that
\[
\m_{0,\beta}=\m_{0,\beta}(1)\in H^{2-\mu(\beta)}(L),
\]
where we have slightly abused notation by using \(\m_{0,\beta}\) both for the operation and for its value on 1.
By Proposition \ref{UD_prop}, we know that nontrivial $\m_{0,\beta}$ satisfies $\mu(\beta)\geqslant 0$ and thus lies either in $H^0(L)$ or in $H^2(L)$. Since $L$ is connected, $H^0(L)\cong \mathbb R$ is generated by the constant-one function \(\one\).

Following \cite[Definition 4.5]{Yuan_unobs}, we define the \emph{superpotential} associated to \(\m\) by
\begin{equation}
	\label{eq_W_m}
	W_\m
	:=\sum_{\mu(\beta)=2} T^{E(\beta)}Y^{\partial\beta}\m_{0,\beta},
\end{equation}
and the \emph{obstruction series} associated to \(\m\) by
\[
Q_\m
:=\sum_{\mu(\beta)=0} T^{E(\beta)}Y^{\partial\beta}\m_{0,\beta}.
\]
We fix a basis $\{\Theta_1,\dots,\Theta_\ell\}$ of $H^2(L)$.
Then, we may write $Q_\m=Q_{\m,1}\Theta_1+\cdots+Q_{\m , \ell}\Theta_\ell$.

Note that \(W_{\m}\) and \(Q_{\m,i}\) are initially only Laurent formal power series, since \(\m\) is, at this stage, merely an abstract $A_\infty$ algebra. 
They restrict to analytic functions on the central fiber $Z$ if and only if $E(\beta)\to\infty$ as $|\partial\beta|\to\infty$ among all $\beta$ with nontrivial $\m_{0,\beta}\neq 0$.
Furthermore, they belong to the germ ring $\mathscr G$ if and only if for every $u\in V$, there exists an open neighborhood $V\subset\mathbb R^n$ of $0$ such that
$E(\beta)+\langle\partial\beta,u\rangle \to \infty$ as $|\partial\beta|\to\infty$ over all contributing classes $\beta$; equivalently, letting $\ell(\partial\beta)$ denote the minimal length of a loop representing $\partial\beta$ with respect to some metric on $L$, there exists a constant $c>0$ such that
\begin{equation}
\label{eq_overconvergent_condition_E_ell}
E(\beta)\;\ge\;c\,\ell(\partial\beta)
\end{equation}
for all $\beta$ with $\m_{0,\beta}\neq 0$.
We remark that, in geometric situations, the condition \eqref{eq_overconvergent_condition_E_ell} corresponds to the well-known reverse isoperimetric inequality \cite{ReverseI,ReverseII}.

\begin{defn}
\label{defn_properly_unobs_overconvergent}
	We call $(H^*(L),\m)$ \textit{overconvergent} if $W_\m$ and $Q_{\m,i}$ define elements of the germ $\mathscr G$.
	We call $(H^*(L),\m)$ \textit{properly unobstructed} if all the $Q_{\m, i}$'s are vanishing (cf. \cite[Definition 4.5]{Yuan_unobs}).
\end{defn}

The standard term in non-archimedean geometry for a series that converges on some neighborhood is ``overconvergent''. Thus, we borrow this terminology here.

Let $\f: (H^*(L), \m) \to (H^*(L), \m')$ be an $A_\infty$ homomorphism in $\UD$.
Following \cite{Yuan_unobs}, we consider $
P_{\f} := \sum_\beta T^{E(\beta)} Y^{\partial\beta} \f_{0,\beta} $.
By degree considerations, every nonzero term \(\f_{0,\beta}\) lies in $H^{1-\mu(\beta)}(L)$. Since $\mu(\beta)\ge 0$ as in Proposition \(\ref{UD_prop}\), it follows that necessarily $\mu(\beta)=0$ and hence \(\f_{0,\beta}\in H^1(L)\).

\begin{defn}
	We call $\f$ \textit{overconvergent} if for $\alpha\in H_1(L)$, the series
$\langle \alpha, P_\f\rangle =\sum_\beta T^{E(\beta)} Y^{\partial\beta} \langle \alpha,\f_{0,\beta}\rangle$
	define an element of the germ ring $\mathscr G$.
\end{defn}

Given an overconvergent $A_\infty$ homomorphism $\f$, the assignment
\[
Y^{\alpha} \mapsto Y^{\alpha} \exp \langle \alpha,  P_\f \rangle
\]
for $\alpha\in  H_1(L)$ naturally gives rise to an algebra endomorphism
\begin{equation}
\label{eq_phi_f}
\phi_\f :  \mathscr G \to\mathscr G . 
\end{equation}
Indeed, since $\langle\alpha, P_\f\rangle$ has supremum norm $<1$ on some neighbourhood of $Z$, the exponential series converges in the non‑archimedean sense.
We call $\phi_\f$ the transitioning homomorphism associated to $\f$.
In general, if $\f$ is just an arbitrary $A_\infty$ homomorphism, we do not know whether $\phi_\f$ admits an inverse.

\subsection{Superpotential functor}
We now relate the the action groupoid introduced in Section \(\ref{s_groupoid_germ}\) to the above setting of $A_\infty$ algebras by constructing the following groupoid and functor.

\begin{defn-thm}
	\label{defn_thm_UD_new}
Define a groupoid $\widehat{\UD}=\widehat{\UD}(L)$ as follows: An object is a properly unobstructed overconvergent $A_\infty$ algebra $(H^*(L),\m)$ in $\UD$ whose underlying graded vector space is fixed to be $H^*(L)$. A morphism is a ud-homotopy class $[\f]$ of an overconvergent $A_\infty$ homomorphism $\f$ in $\UD$ with $\f_{1,0}=\id$.

Moreover, the assignment
\[
W_\bullet :
\widehat{\UD}(L)
\longrightarrow
\mathscr G /\!/ \operatorname{Aut}_Z(\mathscr G),
\qquad
\m \longmapsto W_\m,
\]
which sends an object \(\m\) to its associated superpotential \(W_\m\), and the ud-homotopy class $[\f]$ of a morphism
$\f:\m\to \m'$
to the corresponding homomorphism \(\phi_\f\), defines a contravariant functor of groupoids. In particular, 
\[\phi_\f(W_{\m'})=W_\m.\]
We call it the \textit{superpotential functor} associated to $L$.
\end{defn-thm}

\begin{proof}
This is largely a concise package of the results obtained in \cite{Yuan_unobs} and \cite{Yuan_I_FamilyFloer}.
First, one needs to show that the resulting category is indeed a groupoid. The condition on \(\f_{1,0}\) implies that \(\f\) is an $A_\infty$ quasi-isomorphism. Hence, by the ud-version of the Whitehead theorem, Proposition \(\ref{prop_whitehead}\) (see also \cite[Proposition 2.17]{Yuan_unobs}), \(\f\) admits a homotopy inverse. Therefore every morphism in \(\widehat{\UD}(L)\) is an isomorphism.
One also needs to show that if \(\f\) and \(\g\) are overconvergent $A_\infty$ homomorphisms, then their composition \(\f\diamond \g\), whenever defined, is again overconvergent. Indeed, by \cite[Lemma 4.9]{Yuan_unobs}, we have
$P_{\f\diamond \g}=\f_{1,0}(P_\g)+\phi_\g(P_\f)= P_\g+\phi_\g(P_\f)$.
Thus the components of \(P_{\f\diamond \g}\) define elements of the germ ring $\mathscr G$, since the components of \(P_\f\) and \(P_\g\) do so, and since \(\phi_\g\) maps $\mathscr G$ to itself.

We next check that the assignment on morphisms is well-defined. Namely, if \(\f:\m\to \m'\) is ud-homotopic to \(\f'\), i.e. $\f\simud\f'$, then we need to show that \(\phi_\f=\phi_{\f'}\). By the canceling trick in \cite[Theorem 4.15 and Corollary 4.16]{Yuan_unobs}, the difference \(P_\f-P_{\f'}\) lies in the ideal generated by the obstruction series \(Q_{\m,i}\) associated to \(\m\). Since the objects are assumed to be properly unobstructed (see Definition \(\ref{defn_properly_unobs_overconvergent}\)), this ideal vanishes. Therefore \(\phi_\f=\phi_{\f'}\), and the assignment is well-defined.

The assignment is compatible with identities and composition. Since \(P_{\id}=0\), we have
$\phi_{\id}=\id$.
Moreover, by \cite[Corollary 4.10]{Yuan_unobs}, the assigned maps satisfy
$\phi_{\f\diamond \g}=\phi_{\g}\circ \phi_{\f}$.
Thus the assignment preserves identity morphisms and composition.
Let us also note that each \(\phi_\f\) is an automorphism. Indeed, let \(\g\) be a ud-homotopy inverse of \(\f\). Then
$\g\diamond \f\simud\id$ and $\f \diamond \g \simud\id$.
By the well-definedness and the compatibility with composition just proved, we obtain
$
\phi_\f\circ \phi_\g =
\phi_{\g\diamond \f} =
\phi_{\id} =
\id$
and similarly
$
\phi_\g\circ \phi_\f =
\phi_{\f\diamond \g} =
\phi_{\id} =
\id$.
Hence \(\phi_\f\) is invertible, with inverse \(\phi_\g\).

Finally, to see that the assignment indeed defines a functor to the action groupoid, we need to check that, for every morphism \(\f:\m\to \m'\), the automorphism \(\phi_\f\) sends the target superpotential to the source superpotential, that is,
$
\phi_\f(W_{\m'})=W_\m$.
This is given by the wall-crossing formula established in \cite[Theorem 4.12]{Yuan_unobs}, together with the properly unobstructedness assumption on the objects.
\end{proof}

We remark that, at this stage, the above discussion only concerns suitable abstract $A_\infty$ algebras.

\subsection{Geometric situations}

Let's connect the above framework with concrete symplectic geometry.
Given $J\in\mathfrak J(X)$, let $\check \m=\check \m^{J,\Xi}$ be the $A_\infty$ algebra on $\Omega^*(L)$ constructed in Theorem~\ref{axiom_J_thm} with an auxiliary choice $\Xi$. By Proposition~\ref{prop_minimal_model_homological_perturbation}, the homological perturbation construction with respect to the $g$-harmonic contraction produces a minimal model $A_\infty$ algebra $\m=\m^{g,J,\Xi}$ on $H^*(L)$. It belongs to $\UD$ by Proposition~\ref{prop_minimal_model_UD}. Moreover, the reverse isoperimetric inequality \cite{ReverseI} gives the estimate \eqref{eq_overconvergent_condition_E_ell}. It follows that $\m$ is overconvergent, and that the associated superpotential and obstruction series $W_\m,Q_{\m,j}$ lie in the germ ring $\mathscr G$; see also \cite[Proposition~4.4]{Yuan_unobs}.

\begin{lem}
\label{lem_W_m_not_zero}
In the above context, if \(W_\m \neq 0\) in \(\mathscr G\), then \(L\) bounds at least one \(J\)-holomorphic disk of Maslov index \(2\).
\end{lem}

\begin{proof}
Let \(\mathcal D_J\subset H_2(X,L)\) be the set of nonzero relative homology
classes represented by nonconstant \(J\)-holomorphic disks, i.e.,
the set of $\gamma\in H_2(X,L)\setminus\{0\}$ for which there is a nonconstant \(J\)-holomorphic disk $u:(D,\partial D)\to (X,L)$ with $[u]=\gamma$.
Define
\[
S=S_J
:=
\left\{
\gamma_1+\cdots+\gamma_\ell
\ \middle|\ 
\ell>0 ,\ \gamma_i\in \mathcal D_J
\right\}
\subset H_2(X,L).
\]
Namely, \(S\) is the additive subsemigroup of \(H_2(X,L)\) generated by
\(\mathcal D_J\), without adding the zero element $0\in H_2(X,L)$. Since every class in
\(\mathcal D_J\) has positive symplectic energy and the energy homomorphism is additive,
every element of \(S\) has positive energy.
Besides, there is a constant \(\hbar>0\) such that
	$E(\beta)\ge \hbar$
	for every generator of \(S\) and hence for every nonzero element of \(S\).
By Gromov compactness, for every fixed energy cutoff \(E_0>0\), there are only
	finitely many \(J\)-holomorphic disk classes \(\beta\) with \(E(\beta)<E_0\).
	It follows that the set of possible energy values occurring in \(S\) is discrete
	in \(\mathbb R_{>0}\), and \(S\) is countable. In particular, for any fixed
	\(\beta\in S\), there are only finitely many decompositions
	\[
	\beta=\beta_1+\cdots+\beta_\ell,
	\qquad \beta_i\in S.
	\]
	Indeed, each nonzero summand has energy at least \(\hbar\), so
	\(\ell\le E(\beta)/\hbar\), and Gromov compactness gives only finitely many possible summands below any fixed energy bound.
	
	By Assumption~\ref{assumption_empty_moduli}, the \(A_\infty\) algebra
	\(\check \m=\check \m^{J,\Xi}\) satisfies the following support property:
	if \(\beta\notin S\) and \(\beta\neq 0\), then
	\[
	\check \m_{k,\beta}=0
	\]
	for all \(k\). Let
	\[
	\mi : (H^*(L),\m) \longrightarrow (\Omega^*(L),\check \m)
	\]
	be the \(A_\infty\) quasi-isomorphism obtained from the \(g\)-harmonic
	contraction \((i,\pi,G)\) by homological perturbation, as in
	Proposition~\ref{prop_minimal_model_homological_perturbation}. We recall
	the gappedness condition around \eqref{CC_pi_eq}: if \(E(\beta)\le 0\) and
	\(\beta\neq 0\), then
	$	\mi_{k,\beta}=\check \m_{k,\beta}=\m_{k,\beta}=0$
	for all \(k\). We also recall that \(\mi_{0,0}=0\) by its construction.
	
	We first claim that
	\[
	\mi_{0,\beta}=0
	\qquad
	\text{for every } \beta\notin S.
	\]
It suffices to consider classes of positive energy.
	Suppose, to the contrary, that there exists \(\beta\notin S\) with
	\(\mi_{0,\beta}\neq 0\). Choose such a \(\beta\) with \(E(\beta)>0\) minimal. By the tree induction formula \eqref{eq_induction_tree_minimal_model}, every
	contribution to \(\mi_{0,\beta}\) is of the form
	$
	G \circ \check \m_{\ell,\alpha} \circ 
	(
	\mi_{0,\beta_1} \otimes \cdots \otimes \mi_{0,\beta_\ell}
	)$
	where
	$
	(\ell,\alpha)\neq (1,0)$ and $\alpha+\beta_1+\cdots+\beta_\ell=\beta$.
	Assume that such a term is nonzero. Then
	\(\check \m_{\ell,\alpha}\neq 0\), and hence, by the above support property of
	\(\check \m\), either \(\alpha=0\) or \(\alpha\in S\).
	
	If \(\ell=0\), then the term is \(G\circ \check \m_{0,\beta}\). As
	\(\beta\notin S\), the support property gives \(\check \m_{0,\beta}=0\),
	a contradiction.
	Now assume \(\ell >0 \). Since the term is nonzero, each factor
	\(\mi_{0,\beta_i}\) is nonzero. In particular, \(\beta_i\neq 0\) as
	\(\mi_{0,0}=0\). Also \(E(\beta_i)\ge 0\) by gappedness. If \(\alpha\in S\),
	then \(E(\alpha)>0\) and thus
	$E(\beta_i)<E(\beta)$ for each \(i\). By the minimality of \(E(\beta)\), each \(\beta_i\) must
	belong to \(S\). Hence
	$\beta=\alpha+\beta_1+\cdots+\beta_\ell\in S$,
	contradicting the choice of \(\beta\notin S\).
	It remains to consider the case \(\alpha=0\). Since
	\((\ell,\alpha)\neq (1,0)\), we have \(\ell\neq 1\), and hence
	\(\ell\ge 2\). As above, nonvanishing of the term implies
	\(\mi_{0,\beta_i}\neq 0\) for every \(i\), so each \(\beta_i\neq 0\).
	Thus each \(E(\beta_i)>0\), and since
	$\beta_1+\cdots+\beta_\ell=\beta$, we again have \(E(\beta_i)<E(\beta)\) for every \(i\). By the minimality of
	\(E(\beta)\), each \(\beta_i\in S\). Hence
	$\beta=\beta_1+\cdots+\beta_\ell\in S$,
	again a contradiction. This proves the claim.
	Moreover, the same argument, using the corresponding tree formula \eqref{eq_induction_tree_minimal_model} for the minimal model $A_\infty$ algebra, shows that
	\[
	\m_{0,\beta}=0
	\qquad
	\text{for every } \beta\notin S.
	\]
	
	Now suppose that \(W_\m\neq 0\) in \(\mathscr G\). By the definition of the
	superpotential \(W_\m\), this means that there exists a class
	\(\beta\in H_2(X,L)\) with Maslov index
	$\mu(\beta)=2$
	such that the corresponding coefficient of \(W_\m\), equivalently the
	relevant component of \(\m_{0,\beta}\), is nonzero. By the statement
	just proved, this class \(\beta\) must belong to \(S\).
Therefore \(\beta\) can be written as a finite sum
	$	\beta=\gamma_1+\cdots+\gamma_r$,
	where each \(\gamma_j\in S\) is represented by a nonconstant \(J\)-holomorphic disk.
	Since the Maslov index is additive, we have
	\[
	2=\mu(\beta)=\mu(\gamma_1)+\cdots+\mu(\gamma_r).
	\]
	Under the standing semipositivity assumption on Maslov indices as in Proposition \ref{UD_prop}, each \(\mu(\gamma_j) \ge 0\) is a
	nonnegative even integer. Hence at least one summand satisfies
	$\mu(\gamma_j)=2$.

The corresponding \(J\)-holomorphic disk has Maslov index \(2\). Thus \(L\)
	bounds a \(J\)-holomorphic disk of Maslov index \(2\), as claimed.
\end{proof}

\begin{lem}\label{lem_no_maslov_zero_implies_proper_unobs}
Assume that $L$ does not bound any nonconstant $J$-holomorphic disk of Maslov index 0 (this holds, in particular, when $L$ is monotone). Then the minimal model $A_\infty$ algebra \(\m=\m^{g,J,\Xi}\) is properly unobstructed in the sense of Definition~\ref{defn_properly_unobs_overconvergent}.
\end{lem}

The proof is essentially identical to that of Lemma~\ref{lem_W_m_not_zero}, and we omit it.
Recall that the obstruction series $Q_{\m,j}$ are determined by the terms $\m_{0,\beta}$ with Maslov index $\mu(\beta)=0$. Thus, in practice, the absence of Maslov-index-zero disks gives a simple sufficient criterion for proper unobstructedness.

In summary, in the above context, \((\HL,\m)\) defines an object of the groupoid \(\widehat{\UD}(L)\) introduced in Definition-Theorem~\ref{defn_thm_UD_new}.

Next, we aim to study the effect of making different choices.

Let $\pmb J=\{J_t\}_{t\in[0,1]}$ be a smooth path of $\omega$-tame almost complex structures with $J_0=J$ and $J_1=J'$.
Let \(\check \m'=\check \m^{J',\Xi'}\) be an $A_\infty$ algebra on $\OL$ from Theorem \ref{axiom_J_thm} for some auxiliary choice $\Xi'$.
Using Proposition \ref{prop_minimal_model_homological_perturbation}, let \(\m'=\m^{g',J',\Xi'}\) be the minimal model $A_\infty$ algebra obtained by using the $g'$-harmonic contraction for some metric $g'$.
By Theorem \ref{axiom_J_thm} (b), there exists a pseudo-isotopy \(\check \M:=\check \M^{\pmb J,\pmb \Xi}\) on $\Omega^*(L)_{[0,1]}$ whose restrictions at $t=0$ and $t=1$ are \(\check \m\) and \(\check \m'\) respectively.
Next, choose a smooth path $\pmb g=\{g_t\}_{t\in[0,1]}$ of Riemannian metrics connecting $g_0=g$ to $g_1=g'$. 
Applying the second part of Proposition \ref{prop_minimal_model_UD}, we obtain a pseudo-isotopy $\M:= \M^{\pmb g, \pmb J, \pmb \Xi}$ on $H^*(L)_{\oi}$ whose restrictions at $t=0$ and $t=1$ are \(\m\) and \(\m'\) respectively.
Furthermore, applying Proposition \ref{from_pseudo_to_A_homo_prop} to $\M$ yields an $A_\infty$ quasi-isomorphism 
\[ \f: (H^*(L),\m) \to (H^*(L), \m')
\]
with $\f_{1,0}=\id$.
By \cite[Theorem~1.5]{Yuan_I_FamilyFloer}, proper unobstructedness is preserved under $A_\infty$ quasi-isomorphisms in $\UD$; thus, if $\m$ is properly unobstructed, then so is the $A_\infty$ algebra \(\m'\).

In particular, under the above assumption, \(\f\) is a morphism in the groupoid \(\widehat{\UD}(L)\) of Definition-Theorem~\ref{defn_thm_UD_new}. Therefore,
$\phi_\f(W_{\m'})=W_\m$.
Since \(\phi_\f\) is an automorphism of $\mathscr G$, we conclude that:

\begin{lem}
\label{lem_W_neq_0_diff_choice}
$W_\m\neq 0$ in $\mathscr G$ if and only if $W_{\m'}\neq 0$ in $\mathscr G$.
\end{lem}

Consequently, it is valid to introduce the following definition.

\begin{defn}
	We say that $L$ is \textit{effective} if, for some choice of $g$, $J$, and $\Xi$, the associated minimal model $A_\infty$ algebra
	$\m=\m^{g,J,\Xi}$ on $\HL$ as constructed above
	is properly unobstructed and satisfies $W_{\m}\neq 0$
	in $\mathscr G$. By the preceding discussion, this condition is independent of the choice of $g$, $J$, and $\Xi$.
\end{defn}

Now, we are ready to prove our main result:

\begin{proof}[Proof of Theorem \ref{main_thm}]
	Let $\mathbb S\subset [0,1]$ be the set of parameters $s$ for which $L_s$ is effective. By condition, $0\in\mathbb S$.
	We claim that $\mathbb S$ is both open and closed in $[0,1]$. In fact, we will prove that $\mathbb S$ is locally constant in the following sense: for every $s\in[0,1]$, there exists an interval $I\subset[0,1]$ containing $s$ such that $s\in\mathbb S$ if and only if $s'\in\mathbb S$ for every $s'\in I$.

Fix $s\in[0,1]$, and choose a Weinstein neighborhood $\mathcal U$ of $L_s$. Identify $\mathcal U$ with a neighborhood of the zero section in the cotangent bundle $T^*L_s$, under which $L_s$ is identified with the zero section.
After shrinking $\mathcal U$ if necessary, we may find an interval $I$ containing $s$ so that $L_{s'}\subset \mathcal U$ for every $s'\in I$.
Besides, each $L_{s'}$ is the graph of a small closed one-form $\alpha_{s'}\in \Omega^1(L_s)$, where the family $I\ni s'\mapsto \alpha_{s'}$ is smooth and $\alpha_s=0$. In other words,
$L_{s'}=\Gamma_{\alpha_{s'}}=\{(q,\alpha_{s'}(q))\in T^*L_s:q\in L_s\}$.

Choose a smooth cutoff function
$\rho:T^*L_s \to [0,1]$
whose support is contained in $\mathcal U$, and which is identically $1$ on a slightly smaller neighborhood containing the graphs $\Gamma_{\alpha_{s'}}$ for all $s'\in I$. Define
$F_{s,s'}(q,p)
=
\bigl(q, p+\rho(q,p)\alpha_{s'}(q)\bigr)$
on $\mathcal U\subset T^*L_s$, where $q\in L_s$ and $p\in T_q^*L_s$; meanwhile, we set
$F_{s,s'}(x)=x$
for $x\in X\setminus \operatorname{supp}(\rho)$.
Shrinking $I$ further if necessary, we may assume that the one-forms $\alpha_{s'}$ are uniformly $C^1$-small. Hence the maps $F_{s,s'}$ are $C^1$-close to the identity. In particular, each $F_{s,s'}$ is a diffeomorphism of $X$. Moreover, because $\rho\equiv 1$ near the zero section, we have
$
F_{s,s'}(q,0)
=
(q,\alpha_{s'}(q))$
and therefore
$F_{s,s'}(L_s)=L_{s'}$.

The assignment
\[
F:I\times X\longrightarrow X,\qquad
F(s',x)=F_{s,s'}(x),
\]
is smooth, since $s'\mapsto\alpha_{s'}$ is a smooth family of one-forms and the cutoff function $\rho$ is a fixed smooth function. Since $\alpha_s=0$, we have
$F_{s,s}=\operatorname{id}_X$.
Since $F_{s,s'}$ is $C^1$-close to the identity, the pushforward almost complex structure
\[(F_{s,s'})_*J_s
=
dF_{s,s'}\circ J_s\circ dF_{s,s'}^{-1}
\]
is $C^0$-close to $J_s$. Since the space of $\omega$-tame almost complex structures is open in the $C^0$-topology, after shrinking $I$ if necessary, we may assume that $(F_{s,s'})_*J_s$ is $\omega$-tame for every $s'\in I$.
Since $F_{s,s'}$ is isotopic to the identity, it carries the grading of $L_s$ to the grading of $L_{s'}$ and preserves the Maslov index.

Given $s'\in I$, we use the notation $\beta'=(F_{s,s'})_*\beta$ for every $\beta\in H_2(X,L_s)$.
The diffeomorphism
$F_{s,s'}:(X,L_s) \to (X,L_{s'})$
identifies the corresponding moduli spaces of stable disks; cf. Section \ref{s_axiomatization}. Specifically, for every $\beta\in H_2(X,L_s)$, there is naturally an induced map
\[
\mathcal F_{s,s'}:
\mathcal M_{k+1,\beta}(J_s,L_s)
\longrightarrow
\mathcal M_{k+1,\beta'}
\bigl((F_{s,s'})_*J_s,L_{s'}\bigr)
\]
defined by
\[
\mathcal F_{s,s'}\bigl([\Sigma,u,\mathbf z]\bigr)
=
[\Sigma,F_{s,s'}\circ u,\mathbf z].
\]
Indeed, if
$u:(\Sigma,\partial\Sigma)\to (X,L_s)$
is $J_s$-holomorphic, then
\[
d(F_{s,s'}\circ u)\circ j_\Sigma
=
dF_{s,s'}\circ du\circ j_\Sigma
=
dF_{s,s'}\circ J_s\circ du
=
\bigl((F_{s,s'})_*J_s\bigr)\circ d(F_{s,s'}\circ u).
\]
Thus
$F_{s,s'}\circ u:(\Sigma,\partial\Sigma)\to (X,L_{s'})$
is $(F_{s,s'})_*J_s$-holomorphic and represents the class
$(F_{s,s'})_*\beta\in H_2(X,L_{s'})$. The stability condition is also preserved as the domain is unchanged and $F_{s,s'}$ is a diffeomorphism. The inverse map is induced by $F_{s,s'}^{-1}$ and $\mathcal F_{s,s'}$ is a bijection.
Then, $\mathcal F_{s,s'}$ is a homeomorphism with respect to the stable-map topology and the moduli space system $\mathbb M(J_s,L_s)$ (see \eqref{moduli_space_system_eq}) is componentwise homeomorphic to
$\mathbb M\bigl((F_{s,s'})_*J_s,L_{s'}\bigr)$.

In the context of Theorem \ref{axiom_J_thm}, let $\check{\m}^{J_s,\Xi}$ be an $A_\infty$ algebra on $\Omega^*(L_s)$ obtained from a choice $\Xi\in \mathbb V(J_s;L_s)$ of virtual fundamental chain associated with the moduli space system $\mathbb M(J_s,L_s)$. The above identification of moduli space systems transports this choice to a corresponding choice $\Xi' \in
\mathbb V\bigl((F_{s,s'})_*J_s;L_{s'}\bigr)$
associated with the moduli system $\mathbb M\bigl((F_{s,s'})_*J_s,L_{s'}\bigr)$.
The diffeomorphism $F_{s,s'}$ identifies the $A_\infty$ algebra $\check \m:=\check \m^{J_s,\Xi}$ constructed from $\mathbb M(J_s,L_s)$ with the one $\check \m':=\check \m^{J_{s'}, \Xi'}$ constructed from $\mathbb M\bigl((F_{s,s'})_*J_s,L_{s'}\bigr)$
in the following sense:
Writing
$F^*:=F_{s,s'}^*:\Omega^*(L_{s'}) \to \Omega^*(L_s)$,
for every $\beta\in H_2(X,L_s)$, $\beta'=(F_{s,s'})_*\beta$, and
$y_1,\dots,y_k\in \Omega^*(L_{s'})$, one has
\[
F^*\Bigl(
\check{\m}'_{k,\beta'}(y_1,\dots,y_k)
\Bigr)
=
\check{\m}_{k,\beta}
\bigl(F^*y_1,\dots,F^*y_k\bigr).
\]
This is known as the Fukaya trick; see \cite[Lemma 13.4]{FuCyclic}.

Moreover, for the chosen metric $g$ on $L_s$, let $F_*g$ denote the induced metric on $L_{s'}$, then the corresponding harmonic contractions are related to each other via $F$; see \cite[Lemma 3.13]{Yuan_I_FamilyFloer}.
Accordingly, the minimal model $A_\infty$ algebras
$\m$ and $\m'$
associated respectively with \(\check{\m}\) and \(\check{\m}'\), using the $g$-harmonic contraction on $L_s$ and the $F_*g$-harmonic contraction on $L_{s'}$, satisfy the analogous identity on cohomology:
\[
F^*\Bigl(
\m'_{k,\beta'}(x_1,\dots,x_k)
\Bigr)
=
\m_{k,\beta}
\bigl(F^*x_1,\dots,F^*x_k\bigr),
\]
where $x_1,\dots,x_k\in H^*(L_{s'})$ and $F^*:H^*(L_{s'})\to H^*(L_s)$ is the induced map on cohomology.
This follows directly from the algebraic construction of the minimal model and the compatibility of harmonic contractions under pullback; see \cite[Lemmas 3.14]{Yuan_I_FamilyFloer}.
In particular, if $\beta\in H_2(X,L_s)$ and
$\beta'\in H_2(X,L_{s'})$ are corresponding classes under the identification induced by $F=F_{s,s'}$, then
$F^*\bigl(\m'_{0,\beta'}\bigr)=\m_{0,\beta}$.
If, moreover, $\mu(\beta)=\mu(\beta')=2$, then
$\m'_{0,\beta'}\in H^0(L_{s'})$ and $\m_{0,\beta}\in H^0(L_s)$
by degree reasons. Hence, after identifying $H^0(L_{s'})$ with $H^0(L_s)$
via $F^*$, we have $\m'_{0,\beta'}=\m_{0,\beta}$ viewed as numbers.

Clearly the map $F=F_{s,s'}$ induces an isomorphism $H_1(L_s)\cong H_1(L_{s'})$ and thus an isomorphism between $\Lambda[[ H_1(L_s)]]$ and $\Lambda[[H_1(L_{s'})]]$ defined by
$Y^\gamma \mapsto Y^{F_*\gamma}$.
Let $\{e_i\}$ be a basis of $H_1(L_s)$, and set $e_i':=F_*e_i$. If we write $Y_i:=Y^{e_i}$ and $Y_i':=Y^{e_i'}$,
then this isomorphism is equivalently given by $Y_i \mapsto Y_i'$.
Let $\{e_i^\vee\}$ be the dual basis of $H^1(L_s)$. 
For clarity, we write $[\alpha_{s,s'}]=\sum_i a_i e_i^\vee$.
Similarly, we write
$\partial\beta=\sum_i (\partial_i\beta)e_i$ and $\partial\beta'=\sum_i (\partial_i\beta')e_i'$.
As $\beta'=F_*\beta$, we have $\partial\beta'=F_*(\partial\beta)$ and hence $\partial_i\beta'=\partial_i\beta$ for every $i$.
Observe that
\[
\textstyle
E(\beta') = \int_{\beta'} \omega = \int_\beta F^*\,\omega
= \int_\beta \omega + \int_{\partial\beta} \alpha_{s'}
= E(\beta) + \langle[\alpha_{s'}],\partial\beta\rangle = E(\beta) + \sum_i a_i \ \partial_i\beta .
\]
Then, under the above identification, the superpotential series
\[
W_{\m'}
=
\sum_{\beta'} T^{E(\beta')}Y^{\partial\beta'}\m'_{0,\beta'} = \sum_{\beta'} T^{E(\beta')} \prod_i (Y_i')^{\partial_i\beta} \ \m'_{0,\beta'}
\]
of $\m'$ in $\Lambda[[H_1(L_{s'})]]$ is identified with the series
\begin{align*}
\sum_{\beta} T^{E(\beta)+\sum_i a_i  \partial_i\beta}  \, \prod_i Y_i^{\partial_i\beta}
\ \m_{0,\beta}  = 
\sum_{\beta} T^{E(\beta)}  \, \prod_i (T^{a_i}Y_i)^{\partial_i\beta}
\ \m_{0,\beta} 
\end{align*}
in $\Lambda[[H_1(L_s)]]$.
In other words, after identifying \(\Lambda[[H_1(L_{s'})]]\) with
\(\Lambda[[H_1(L_s)]]\) as above, the series \(W_{\m'}\) is identified with the pullback of the superpotential
$W_{\m}$ in $\Lambda[[H_1(L_s)]]$
under the coordinate change
$Y_i\mapsto T^{a_i}Y_i$.
In particular, \(W_{\m}\) is nonzero if and only if \(W_{\m'}\) is nonzero.
In view of Lemma \ref{lem_W_neq_0_diff_choice}, this implies that $s\in\mathbb S \Longleftrightarrow s'\in\mathbb S$ for every \(s'\in I\). Hence either \(I\subset\mathbb S\) or \(I\cap\mathbb S=\varnothing\) as desired.

Finally, using Lemma \ref{lem_W_m_not_zero} and \ref{lem_W_neq_0_diff_choice} completes the proof.
\end{proof}

\bibliographystyle{alpha}
\bibliography{mybib_exist}		
	
\end{document}